\documentclass[11pt]{amsart}
\usepackage[marginratio=1:1,height=614pt,width=400pt,tmargin=117pt]{geometry}
\usepackage{amsmath,amsthm,amssymb}
\usepackage{comment}
\usepackage{enumitem}
\usepackage{cleveref}
\usepackage{mathtools}
\usepackage[final]{pdfpages}

\newtheorem{theorem}{Theorem}[section]
\newtheorem{lemma}[theorem]{Lemma}
\newtheorem{proposition}[theorem]{Proposition}
\newtheorem{corollary}[theorem]{Corollary}

\theoremstyle{definition}
\newtheorem{definition}[theorem]{Definition}

\theoremstyle{remark}
\newtheorem{remark}[theorem]{Remark}

\numberwithin{equation}{section}

\newcommand{\R}{\mathbb{R}}
\newcommand{\N}{\mathbb{N}}
\newcommand{\Z}{\mathbb{Z}}

\DeclareMathOperator{\divr}{div}

\DeclareMathOperator{\supp}{supp}

\DeclareMathOperator{\diam}{diam}

\begin{document}

\title[Partially regular weak solutions of the Navier-Stokes equations]{Partially regular weak solutions of the Navier-Stokes equations in $\R^4 \times [0,\infty[$}

\author{Bian Wu}
\address{ETH Z\"urich, R\"amistrasse 101, 8092 Z\"urich, Switzerland}
\email{bian.wu@math.ethz.ch}

\maketitle

\begin{abstract}
We show that for any given solenoidal initial data in $L^2$ and any solenoidal external force in $L_{\text{loc}}^q \bigcap L^{3/2}$ with $q>3$, there exist partially regular weak solutions of the Navier-Stokes equations in $\R^4 \times [0,\infty[$ which satisfy certain local energy inequalities and whose singular sets have locally finite $2$-dimensional parabolic Hausdorff measure. With the help of a parabolic concentration-compactness theorem we are able to overcome the possible lack of compactness arising in the spatially $4$-dimensional setting by using defect measures, which we then incorporate into the partial regularity theory.
\end{abstract}

\section{The Navier-Stokes equations} \label{intro}

\subsection{Introduction}

The nonstationary Navier-Stokes equations governing the motion of an incompressible viscous fluid in $\R^n \times [0,T]$ with $T>0$ are given by
\begin{equation} \label{navierstokes}
  \begin{aligned}
  \partial_t u - \Delta u + (u \cdot \nabla) u + \nabla p &= f \quad (x,t) \in \R^n \times [0,T] \\
  \divr u &= 0
  \end{aligned}
\end{equation}
with initial condition $u(x,0)=u_0(x), u_0 \in L^2(\R^n)$. Note that is suffices to consider weakly solenoidal forces $f$. Indeed, by Helmholtz-Weyl decomposition, for any external force $f \in L^q(\R^n)$, $q>1$, we have a decomposition $f=f_s+f_{p'}$ such that $\divr f_s = 0$ and $f_{p'} = \nabla p'$ for some $p' \in W^{1,q}(\R^n)$. We can insert the component $f_{p'}$ into the pressure term $\nabla p$.\par

The existence and regularity problem of the Navier-Stokes equations is one of the most significant open questions in the field of partial differential equations. The case of $\R^2$ was known to Leray \cite{leray1933etude} in 1933. Later, the case of $2$D domains with boundaries was settled by Ladyzhenskaya \cite{ladyzhenskaia1960solution} in 1959. The case $n=3$ is one of the millennium problems and is still open. However, remarkable progress has been made since the pioneering work by Leray in 1930s. Leray \cite{leray1934mouvement} and Hopf \cite{hopf1950anfangswertaufgabe} proved the existence of weak solutions of these equations in dimensions $n \geq 2$ in the whole space and on bounded open domains with smooth boundary in 1934 and 1950, respectively. These weak solutions, called Leray-Hopf weak solutions, satisfy \eqref{navierstokes} in the distributional sense and belong to $L_t^\infty L^2_x \bigcap L^2_tH_x^1(\R^n \times [0,T])$. Leray-Hopf weak solutions also satisfy the following global energy inequality
\begin{equation}
    \frac{1}{2} \|u(\cdot,t)\|^2_{L^2(\R^n)} + \int_0^t\int_{\R^n} |\nabla u|^2 dxdt \leq \frac{1}{2} \|u_0\|^2_{L^2(\R^n)}.
\end{equation}
For Leray-Hopf weak solutions, this inequality can be obtained from a Galerkin approximation with the help of weak convergence results, instead of testing \eqref{navierstokes} with $u$. Indeed, we only have $(u \cdot \nabla)u \in L^{(n+2)/(n+1)}$ because of the embedding $L_t^\infty L^2_x \bigcap L^2_tH_x^1 \subset L_{t,x}^{2+4/n}$, and the product of this term with $u \in L_t^\infty L^2_x \bigcap L^2_tH_x^1$ is not necessarily integrable. \par

An important step towards a better understanding of the Navier-Stokes equations in dimension $n=3$ was made by Scheffer \cite{scheffer1977hausdorff,scheffer1978navier,scheffer1980navier} and Caffarelli, Kohn and Nirenberg \cite{caffarelli1982partial}. In \cite{scheffer1980navier} Scheffer pioneered the partial regularity theory by introducing the notion of suitable weak solutions and proving their existence in dimension $n=3$ when $f=0$. Moreover, he proved that the singular sets of these suitable weak solutions have finite $\frac{5}{3}$-dimensional Hausdorff measure in space-time. In \cite{scheffer1978navier}, Scheffer showed that in dimension $n=4$, there exist weak solutions whose singular sets have finite $3$-dimensional Hausdorff measure in space-time. Caffarelli, Kohn and Nirenberg made remarkable improvements and generalizations in dimension $n=3$ by proving local partial regularity results for a general force and by proving that the $1$-dimensional parabolic Hausdorff measure of the singular sets of suitable weak solutions is zero.
\par

The suitable weak solutions are distributional solutions in the class $L_t^\infty L^2_x \bigcap L^2_tH_x^1$ which satisfy a local energy inequality, i.e., for any $-r_0^2 < t < 0$ and any scalar function $0 \leq \phi \in C^{\infty}(Q_{r_0})$ with $\phi = 0$ in $Q_{r_0} \backslash Q_{r_1}$ and $\phi = 1$ in $Q_{r_2}$ for any $0<r_2<r_1<r_0$, the following inequality holds,
\begin{equation} \label{localenergy}
  \begin{split} 
  \int_{Q_{r_0} \times \{t\}} |u|^2 \phi dx &+ 2 \int^0_{-r_0^2} \int_{Q_{r_0}} |\nabla u|^2 \phi dxds \\
  &\leq \int^0_{-r_0^2} \int_{Q_{r_0}} \Big( |u|^2 ( \partial_t \phi + \Delta \phi ) + (|u|^2+2p)u \cdot \nabla \phi + 2f \cdot u \Big) dxds.
  \end{split}
\end{equation}
For a general function in $L_t^\infty L^2_x \bigcap L^2_tH_x^1$, this first integral may not be well-defined. However, for Leray-Hopf weak solutions, one can redefine the solution in those time slices with zero measure in $t$ and make this solution weakly continuous as a mapping from time interval to $L^2_x$. Note that it is unknown if Leray-Hopf weak solutions satisfy the local energy inequality, since $u\phi$ is not an admissible test function. Caffarelli, Kohn and Nirenberg \cite{caffarelli1982partial} proved their existence in a general setting by retarded mollification, then they showed that the sequence of the approximation solutions $\{u_k\}_{k \in \N}$ of the mollified equations is relatively compact in $L_{t,x}^3$-topology. Because we have $L_t^\infty L^2_x \bigcap L^2_tH_x^1 \subset L_{t,x}^{2+4/n}$, the compactness in $L_{t,x}^3$ in dimension $n=3$ can be obtained from boundedness of $\{u_k\}_{k \in \N}$ in $L_{t,x}^{10/3}$ and compactness in $L^2_{t,x}$. However, in dimension $n=4$ we only have that the approximation sequence is relatively compact in $L_{t,x}^\alpha$ for $\alpha<3$, which is not enough for the local energy inequality to hold in the limit. As the local energy inequality is the most important ingredient for partial regularity theory, the following natural question arises, mentioned by Dong and Du in Remark 1.1 of \cite{dong2007partial}. \par

\textbf{Open question:} Do there exist partially regular weak solutions of the Navier-Stokes equations in $4$D which satisfy certain local energy estimates? \par

This problem has not been answered for quite a long time; however, there have been many results in this direction. Early in 1978, Scheffer \cite{scheffer1978navier} constructed weak solutions $u$ in $\R^4 \times [0,+\infty)$ which are continuous outside a closed set of finite $3$-dimensional Hausdorff measure. We remark that Scheffer's weak solutions do not necessarily satisfy the local energy inequality, because of the loss of compactness mentioned above. Dong and Du \cite{dong2007partial} showed that the 2-dimensional Hausdorff measure of the singular sets of local-in-time regular weak solutions at the first blow-up time is zero. Under the assumption on the existence of the suitable weak solutions, Dong, Gu \cite{dong2014partial} and Wang, Wu \cite{wang2014unified} independently proved that the $2$-dimensional parabolic Hausdorff measure of the singular sets is zero. A similar study of partial regularity has also been carried out for the magneto-hydrodynamic equations by Choe and Yang \cite{choe2015hausdorff}. In the direction of the local energy inequality, Biryuk, Craig and Ibrahim \cite{biryuk2007construction} discussed the difficulty of validating the local energy inequality in higher dimensions $n \geq 4$. Taniuchi \cite{taniuchi1997generalized} proved the local energy inequality in the dimensions $3 \leq n \leq 10$, given some conditional regularity on distributional solutions. \par

\subsection{New observations, main result, and the organization of this paper}
The aim of this paper is to answer the open question stated above, by constructing weak solutions of the Navier-Stokes equations in $4$D satisfying the local energy inequalities \eqref{evo_local_energy_1} and \eqref{evo_local_energy_2} below and showing that these solutions are global-in-time partially regular with singular sets of finite $2$-dimensional parabolic Hausdorff measure. Thus, we improve Scheffer's result in \cite{scheffer1978navier} by refining the estimate of the Hausdorff dimension of singular sets from 3 to 2 and allowing general forces. We remark that the local energy inequalities \eqref{evo_local_energy_1} and \eqref{evo_local_energy_2} are slightly weaker than the local energy inequality \eqref{localenergy}. Nevertheless, they suffice to give all the partial regularity criteria which Caffarelli, Kohn and Nirenberg have obtained for $3$D case. \par

As we have discussed before, the $L^3_{t,x}$-norm is critical for the local energy inequality (1.3) in dimension $n=4$, in the sense that we need to deal with a possible loss of compactness of smooth approximating sequences in this norm. For this reason, we develop a parabolic concentration-compactness method to study concentration phenomena in a space-time topology. Our two new observations are as follows. 
\begin{enumerate}[leftmargin=*]
  \item The measures $|u_k|^3dxdt$ induced by the solutions $u_k$ of the regularized equations \eqref{reg_navierstokes} below are relatively compact in the sense of measures. 
  \item The limit measures have the same scaling properties as classical solutions of \eqref{navierstokes} and satisfy the local energy inequalities \eqref{evo_local_energy_1} and \eqref{evo_local_energy_2}.
\end{enumerate}
\par
With these ingredients, we are able to estimate the concentration locally and construct solutions satifying local energy estimates involving concentration measures. To couple these measures and our weak solutions, we introduce a new notion of generalized solution, namely the notion of weak solution set in \Cref{weaksolutionset_nonsta}. Finally, we use the iteration scheme as Caffarelli, Kohn and Nirenberg \cite{caffarelli1982partial} to show that the functions $(u,p)$ in the weak solution set that we construct are partially regular. \par

Our main theorem is stated for $\R^4$ as follows. However, the method we use is robust and it also applies to more general open domains, for instance, bounded open domains with smooth boundary.

 \begin{theorem} \label{maintheorem_nonsta}
Fix a possibly large time $T>0$ and let $D:=\R^4 \times [0,T]$. Given a weakly solenoidal force $f \in L_{\text{loc}}^q \bigcap L^{3/2}(D)$, $q>3$, there exists a weak solution set $(u,p,\lambda,\omega)$ for the nonstationary Navier-Stokes equations \eqref{navierstokes} in $D$ which satisfies the local energy inequalities \eqref{evo_local_energy_1} and \eqref{evo_local_energy_2}. Moreover, $(u,p)$ is a weak solution of the Navier-Stokes equations with $u \in L_t^{\infty}L_x^2 \bigcap L_t^2H_x^1 (D)$ and $p \in L^{3/2}(D)$, and the singular set $S$ of $u$ as defined in \Cref{singularset_nonsta} satifies $\mathcal{P}^2(S) < \infty$.
\end{theorem}

\subsection{Connection with the stationary Navier-Stokes equations}
 Note that in the sense of energy estimates, the nonstationary Navier-Stokes equations in ${\R^n \times [0,\infty[}$ are similar to the stationary equations in $\R^{n+2}$. The stationary Navier-Stokes equations are given by
\begin{equation} \label{sta_navierstokes}
  \begin{split}
  - \Delta u + (u \cdot \nabla) u + \nabla p &= f \quad x \in \R^{n+2} \\
  \divr u &= 0.
  \end{split}
\end{equation}
The analogue of the $3$D nonstationary case is the $5$D stationary case. In 1988, Struwe proved partial regularity for the stationary case in $\R^5$ in \cite{struwe1988partial}. The existence of regular solutions in $\R^5$ and in periodic domain $\R^5/\Z^5$ was independently proved by Struwe in \cite{struwe1995regular} and Frehse, R\accent23 u\v zi\v cka in \cite{frehse1995existence}. Later, for $f \in L^{\infty}$ and bounded subdomains of $\R^6$, Frehse and R\accent23 u\v zi\v cka proved the existence of regular solutions in \cite{frehse1994regularity1} and \cite{frehse1996existence}. For general open subdomains of $\R^6$ or for unbounded forces $f$, the existence of regular solutions or suitable weak solutions to \eqref{sta_navierstokes} is still open.
\par
The strategy that we use to show \Cref{maintheorem_nonsta} can also be applied to the stationary Navier-Stokes equations in dimension $6$, because the measures $|u_k|^3dx$ induced by the solutions $u_k$ of the regularized equations are relatively compact, modulo the mass vanishing at infinity. However, the stationary $6$D case and the nonstationary $4$D case also differ in some interesting aspects. Although one time dimension counts for two space dimensions, the weak solutions are less regular in time than in space. Consequently, we may say the nonstationary case in space dimension $n=4$ is less regular than the stationary case in $\R^6$. This slight regularity gap actually leads to a manifest difference in the concentration of $L^3_{t,x}$ mass. For instance, we only have point concentration in the stationary case, but line concentration might occur in the nonstationary case. The detailed discussion of the stationary case in $6$D will be given in a separate paper.

\subsection{Connection with variational problems}
The concentration-compactness principle developed by Lions in \cite{lions1985concentration} has been shown to be an effective tool for dealing with elliptic PDEs and variational problems. Basically, this tool may help us understand the process of passing to a weak limit in many cases. By lower semi-continuity, if $x_n \rightharpoonup x$ in a Banach space $X$, we have
\begin{equation} \label{con_var_eq0} \|x\|_X \leq \liminf_{k \rightarrow +\infty} \|x_n\|_X. \end{equation}
Usually, one would like to know if equality holds and if not, why equality fails. \par

A classical example is the existence of extremal functions for Sobolev embeddings. This amounts to find a function such that the following embedding inequality holds with equality,
\begin{equation} \label{con_var_eq1} 
  S\|u\|_{L^{nl/(n-kl)}(\R^n)} \leq \|u\|_{W^{k,l}(\R^n)} \quad k \in \N, l \geq 1,
\end{equation}
where $kl<n$ and $S>0$ is the maximal Sobolev constant.
\par
After suitable translations and dilations, a minimizing sequence in this problem incurs no concentration, thus is relatively compact in $L^{nl/(n-kl)}(\R^n)$. \par
However, the Navier-Stokes equations are not known to admit a variational structure, and translations or dilations will change the external forces and boundary conditions, so it is not possible to normalize solutions using either of these tools. However, Gallagher, Koch and Planchon \cite{gallagher2016blow} used profile decomposition to show that a local-in-time solution of the Navier-Stokes equations which develops a singularity at finite time must blow up in scale-invariant norms at this time, where they also study the residual term in weak convergence. Nevertheless, this approach does not work for our purpose because of the issues concerning controlling external forces, initial conditions and boundary conditions. Despite these technical difficulties that naturally arise in fluid dynamics equations, we shall see that possible concentration loss of $L^3_{t,x}$ mass either can be controlled or causes no harm to the regularity theory we aim to pursue.

\subsection{Connection with other PDEs}
Apart from the stationary Navier-Stokes equations in $\R^6$, our strategy may be applied in a wide class of PDEs. For instance, the following incompressible magneto-hydrodynamic equations have a structure similar to the Navier-Stokes equations,
\begin{equation} \label{magnavier}
  \begin{split}
  \partial_t u - \Delta u + (u \cdot \nabla) u + \nabla p &= (h \cdot \nabla) h \\
  \divr u &= 0 \\
  \partial_t h - \Delta h + (u \cdot \nabla)h - (h \cdot \nabla)u &= 0 \\
  \divr h &= 0.
  \end{split}
\end{equation}
Gu \cite{gu2015regularity} obtained some partial regularity criteria for suitable weak solutions to \eqref{magnavier} in space dimension $4$ assuming that these solutions exist. It is likely that one can construct partially regular weak solutions of the incompressible magneto-hydrodynamic equations with our strategy, but we do not pursue it here.

\section{Weak solution sets and local energy inequalities} \label{local_energy_nonstationary}

To construct weak solutions for the Navier-Stokes equations in the space-time domain $D:=\R^4 \times [0,T]$, we consider the regularized Navier-Stokes equations, namely
\begin{equation} \label{reg_navierstokes}
  \begin{split}
  \partial_t u_k - \Delta u_k + [(\chi_k * u_k) \cdot \nabla] u_k + \nabla p_k &= f \\
  \divr u_k &= 0 \\
  u(\cdot, 0) &= u_0,
  \end{split}
\end{equation}
where $\{\chi_k\}_{k \in \N} \subset C_c^{\infty}(\R^4)$ are standard mollifiers. This regularization was used by Hopf in \cite{hopf1950anfangswertaufgabe} to show the existence of weak solutions. The first step is to use a Galerkin method to construct stronger solutions for the regularized equations with uniform energy estimates. In this section, we also prove that the measures induced by $|u_k|^3dxdt$ in the approximation sequence are relatively compact in the sense of measures. Next, we set up a parabolic concentration-compactness framework.  With all these ingredients and careful estimates on the pressure $p$, we can construct weak solution sets satisfying certain local energy estimates.

\subsection{Solving regularized equations and weak compactness of measures}
Before we prove the existence of weak solutions of the regularized Navier-Stokes equations \eqref{reg_navierstokes}, we recall the following definition of distributional solutions of the Navier-Stokes equations \eqref{navierstokes}, which also avoids any possible ambiguity of realizing the initial data $u_0$.
\begin{definition}
A pair of functions $(u,p) \in L^2_tH_{x,\text{loc}}^1(D) \times L_{\text{loc}}^{1+2/n}(D)$ are distributional solutions of \eqref{navierstokes} if $u$ is weakly divergence-free and for any $\varphi \in C_c^{\infty}(\R^4 \times [0,T])$ and any $t \in [0,T]$, we have
\begin{equation} \label{weaksolution_eq0}
  \begin{split}
    &-\int_0^t\int_{\R^n} u^i\partial_t\varphi_i dxdt + \int_0^t\int_{\R^n} \partial_j u^i \partial_j \varphi_i dxdt + \int_0^t\int_{\R^n} u^j \partial_j u^i \varphi_i dxdt \\
    &- \int_0^t\int_{\R^n} p \divr \varphi dxdt - \int_0^t\int_{\R^n} f_i \varphi_i dxdt = \int_{\R^n} \big(u_0 \cdot \varphi(0) - u(t) \cdot \varphi(t)\big) dx.
  \end{split}
\end{equation}
\begin{remark} \label{rmk_general_ns}
If we restrict the test functions to divergence-free functions, then we have a weak formulation of \eqref{navierstokes} without $p$. If we test with $\nabla \phi$ where $\phi$ is a scalar function, then we obtain the following well-known elliptic equation for the pressure
\begin{equation} \label{pressure_equ} - \Delta p = \partial_i \partial_j (u^i u^j). \end{equation}
\end{remark}
\par
Distributional solutions for the regularized Navier-Stokes equations \eqref{reg_navierstokes} can be defined in a similar way. We now show that there exist weak solutions for \eqref{reg_navierstokes} with uniform energy bounds and satisfying local energy inequality.
\end{definition}
\begin{lemma} \label{preparation_existence_nonsta}
Let $\{\chi_k\}_{k \in \N}$ be a sequence of standard mollifiers and $f \in L^{3/2}(D)$, then we have a sequence $\{(u_k,p_k)\}_{k \in \N} \subset L_t^{\infty} L_x^2 \bigcap L_t^2 H_x^1(D) \times L^{3/2}(D)$ such that $(u_k,p_k)$ is a distributional solution to the regularized nonstationary Navier-Stokes equations \eqref{reg_navierstokes}. Moreover,
\begin{enumerate}[leftmargin=*]
  \item $\{u_k\}_{k \in \N}$ is uniformly bounded in $L_t^{\infty} L_x^2 \bigcap L_t^2 \dot{H}_x^1(D)$,
  \item $\{p_k\}_{k \in \N}$ is uniformly bounded in $L^{3/2}(D)$,
  \item $\{\partial_t u_k\}_{k \in \N}$ is uniformly bounded in $L_t^1 \mathbb{H}_{x,\text{loc}}^*(D)$,
\end{enumerate}
where $\mathbb{H}(\R^4) := \{ \varphi \in H^1(\R^4) | \divr \varphi = 0 \}$ and $\mathbb{H}^*(\R^4)$ is the dual space of $\mathbb{H}(\R^4)$.
Thereby we can pass to the weak limit,
\begin{equation} \label{nonstationary_convergence}
  \begin{split}
  u_k \rightarrow u \quad &\text{weakly in } L_t^2 H_x^1, \\
  u_k \rightarrow u \quad &\text{weakly}-* \text{ in } L_t^{\infty} L_x^2, \\
  p_k \rightarrow p \quad &\text{weakly in } L^{3/2}.
  \end{split}
\end{equation}
This sequence satisfies the local energy inequality, i.e. for any bounded smooth function $\phi$ with bounded derivatives,
\begin{equation} \label{evo_local_energy_seq}
  \begin{split}
  \int_{\R^4} |u_k(t)|^2 \phi(t)dx - \int_{\R^4} |u_0|^2 \phi(0) dx
    + \int_0^t \int_{\R^4} |\nabla u_k|^2 \phi dxdt &\\
  \leq \int_0^t \int_{\R^4} |u_k|^2 |\partial_t \phi + \Delta \phi| dxdt
    + \int_0^t \int_{\R^4} |u_k|^2 &(\tilde{u}_k \cdot  \nabla)\phi dxdt \\
      + \int_0^t \int_{\R^4} 2 p_k (u_k \cdot \nabla) \phi dx + \int_0^t &\int_{\R^4} f \cdot u_k \phi dxdt,
  \end{split}
\end{equation}
where $\tilde{u}_k := \chi_k * u_k$.
\end{lemma}

\begin{remark} \label{rmk_admissiblet} We do not need $\phi$ to have compact support. $\chi_k * u_k$ is bounded and smooth for every $k$. For any bounded function $\phi$ with bounded derivatives, $u_k \phi$ is in $L_t^{\infty} L_x^2 \bigcap L_t^2 H_x^1(D) \subset L^3(D)$; thus, $u_k \phi$ is an admissible test function. \end{remark}

\begin{proof}
The existence of $u \in L_t^{\infty} L_x^2 \cap L_t^2 H_x^1(D)$ can be proved by a standard Galerkin method. We refer to Theorem 4.4 and Theorem 14.1 in \cite{robinson2016three} for an exposition. The existence of $p \in L^{3/2}(D)$ is obtained by $L^p$-theory of elliptic operators and Calderon-Zygmund theory. \par
For the rest, note that in the regularized equations $u_k$ and $u_k \phi$ are admissible test functions. Testing with $u_k$ and $u_k \phi$ yields the uniform boundedness of $\{u_k\}_{k \in \N}$ and $\{p_k\}_{k \in \N}$ and the local energy inequality \eqref{evo_local_energy_seq}. \par
For the uniform boundedness of $\{\partial_t u_k\}_{k \in \N}$, we remark that the weak formulation \eqref{weaksolution_eq0} for $u_k$ is equivalent to
\[ \forall \hspace{1mm} \xi \in \mathbb{H}_x, \hspace{2mm}
  \langle \partial_t u_k, \xi\rangle_{\mathbb{H}_x^* \times \mathbb{H}_x}
    = -\int_{\R^4} \Big( \partial_j u_k^i \partial_j \xi_i + u_k^j \xi_i \partial_j u_k^i - f_i \xi_i \Big) dx \]
for almost every $t \in [0,T]$.
\par
For every $\xi \in C_c^{\infty}(\Omega)$, $\Omega \subset\joinrel\subset \R^4$ and almost every $t \in [0,T]$, we can estimate
\[
    \begin{split}
    \Big| \int_{\R^4} \Big( \partial_j u_k^i \partial_j \xi_i + u_k^j \xi_i \partial_j u_k^i - f_i \xi_i \Big) dx \Big|
    \leq& \|\nabla u_k\|_{L^2_x} \|\nabla \xi\|_{L^2_x}
       + \| \xi \|_{L^3_x} \| f \|_{L^{3/2}_x} \\
      &+ \|\nabla u_k\|_{L^2_x} \|\xi\|_{L^4_x} \|u_k\|_{L^4_x} \\
    \leq& C \big( \|u_k\|^2_{\dot{H}^1_x} + \|u_k\|_{\dot{H}^1_x} + \| f \|_{L^{3/2}_x} \big) \| \xi \|_{H^1_x}.
    \end{split}
\]
Then integrating in time yields that $\{\partial_t u_k\}_{k \in \N}$ is uniformly bounded in $L_t^1 \mathbb{H}_{x,\text{loc}}^*(D)$.
\end{proof}

Next, we prove that certain measures in the limit are relatively compact, which yields a crucial requirement in our parabolic concentration-compactness framework. This is equivalent to tightness of measures. Recall that a collection of measures $\{\mu_i\}_{i \in \Lambda}$ on $\R^n$ is called \textit{tight} if for any $\epsilon>0$, there exists a compact set $K_\epsilon \subset \R^n$ such that $\mu_i(\R^n \backslash K_\epsilon) < \epsilon$ for any $i \in \Lambda$.

\begin{lemma} \label{evo_tightness}
Let the assumptions be as in \Cref{preparation_existence_nonsta}, then $\{|\nabla u_k|^2dxdt\}_{k \in \N}$, $\{|u_k|^2dxdt\}_{k \in \N}$ and $\{|u_k|^3dxdt\}_{k \in \N}$ are tight in the sense of measures.
\end{lemma}

\begin{proof}
We define a cut-off function $\xi \in C^{\infty}(\R^4)$ with bounded derivatives by
\[ 0 \leq \xi \leq 1, \quad \xi|_{B_{\rho}} = 0, 
    \quad \xi|_{\R^4 \backslash B_{2\rho}} = 1, \quad |\nabla \xi| \leq C\rho^{-1}, \quad |\nabla^2 \xi| \leq C\rho^{-2}. \]
and let $\rho > 0$ to be determined. From \Cref{rmk_admissiblet}, we know $u_k \xi$ is an admissible test function. Testing the regularized Navier-Stokes equations \eqref{reg_navierstokes} with $u_k \xi$ yields
\[
  \begin{split}
  \sup_t \int_{\R^4} |u_k(t)|^2 \xi dx - \int_{\R^4} |u_0|^2 \xi dx &+ \iint_D |\nabla u_k|^2 \xi dxdt 
  \leq \iint_D |u_k|^2 |\Delta \xi| dxdt \\ + \iint_D |u_k|^2 |\tilde{u}_k| |\nabla \xi| dxdt
      &+ \iint_D 2 |p_k u_k| |\nabla \xi| dxdt + \iint_D |f \cdot u_k| \xi dxdt.
  \end{split}
\]
The bounds for $\xi$ gives
\[
  \begin{split}
  \sup_t \int_{\R^4 \backslash B_{2\rho}} |u_k(t)|^2 dx 
      &- \int_{\R^4 \backslash B_{\rho}} |u_0|^2 dx 
      + \iint_{D^c_{2\rho}} |\nabla u_k|^2 dxdt \\
    \leq& C \rho^{-2} \iint_{D_{\rho,2\rho}} |u_k|^2 dxdt
      + C \rho^{-1} \iint_{D_{\rho,2\rho}} |u_k|^2 |\tilde{u}_k| dxdt \\
    &+ 2 C \rho^{-1} \iint_{D_{\rho,2\rho}} |p_ku_k| dxdt + \iint_{D_{\rho,2\rho}} |f| |u_k| dxdt \\
    \leq& C \rho^{-2/3}T^{1/3} + C \|f\|_{L^{3/2}(D_{\rho,2\rho})},
  \end{split}
\]
where $D^c_{2\rho}:=(\R^4 \backslash B_{2\rho}) \times [0,T]$, and $D_{\rho,2\rho}:=(B_{2\rho} \backslash B_{\rho}) \times [0,T]$. The second inequality follows from H{\"o}lder inequality and the fact that $\{u_k\}_{k \in \N}$ is uniformly bounded in $L^3(D)$. Finally, letting $\rho$ be arbitrarily large concludes the tightness of $\{|\nabla u_k|^2 dxdt\}_{k \in \N}$. This also implies that $\{|u_k(t)|^2dx\}_{k \in \N}$ is tight uniformly in $t$, which leads to the tightness of $\{|u_k|^2dxdt\}_{k \in \N}$.
\par
For the tightness of the measures $\{|u_k|^3dxdt\}_{k \in \N}$, we use Sobolev inequality and the same cutoff function,
\[
  \begin{split}
  \iint_D |u_k \xi|^3 dxdt & \leq \|u_k \xi\|_{L_t^{\infty}L_x^2} \iint_D |\nabla(u_k \xi)|^2 dxdt \\
  & \leq 2 \|u_k \xi\|_{L_t^{\infty}L_x^2} \Big( \iint_D |\nabla u_k|^2 |\xi|^2 dxdt + \iint_D |u_k|^2 |\nabla \xi|^2 dxdt \Big) \\
  & \leq 2 \|u_k \xi\|_{L_t^{\infty}L_x^2} \Big[ \iint_D |\nabla u_k|^2 |\xi|^2 dxdt + C \rho^{-2/3}T^{1/3} \Big(\iint_D |u_k|^3 dxdt\Big)^{2/3} \Big].
  \end{split}
\]
Given the tightness of $\{|\nabla u_k|^2 dxdt\}_{k \in \N}$ and uniform boundedness of $u_k$ in the natural energy space, arbitrarily large $\rho$ yields the tightness of $\{|u_k|^3dxdt\}_{k \in \N}$.
\end{proof}

With the tightness of the measures, we obtain convergence of $\{u_k\}_{k \in \N}$ in $L^2(D)$.
\begin{lemma} \label{strong_conv}
Let the assumptions be as in \Cref{preparation_existence_nonsta}. The sequence $\{u_k\}_{k \in \N}$ is relatively compact in $L^2(D)$. Consequently, the weak limit $(u,p)$ are distributional solutions of the Navier-Stokes equations \eqref{navierstokes}.
\end{lemma}

\Cref{strong_conv} is a direct consequence of the bounds in \Cref{preparation_existence_nonsta} and the following compactness result.
\begin{lemma}[Corollary 6, Simon, \cite{simon1986compact}] \label{AubinLions}
Let $X,Y,B$ be Banach spaces and assume that we have the embeddings
\[ X \xhookrightarrow{\text{compact}} B \hookrightarrow Y. \]
If a sequence $\{u_k\}_{k \in \N}$ is bounded in $L^\alpha(0,T,B) \bigcap L^1_{\text{loc}}(0,T,X), \alpha \in (1, +\infty]$ and $\partial_t u_n$ is bounded in $L^1_{\text{loc}}(0,T,Y)$, then there exists a subsequence of $\{u_k\}_{k \in \N}$ which converges strongly in $L^\beta(0,T,B)$ for any $\beta \in [1,\alpha)$.
\end{lemma}

\begin{proof}[Proof of \Cref{strong_conv}]
Note that, to get a compact Sobolev embedding, we first restrict to $B_l \subset \R^4$ with $l \in \N^*$, then letting $\alpha = +\infty$ and
\[ X = H^1(B_l), \quad B = L^2(B_l), \quad Y = \mathbb{H}^*(B_l) \]
give the strong convergence of a subsequence of $\{u_k\}_{k \in \N}$ in $L^2(B_l\times [0,T])$. By enlarging $r$ to infinity, a diagonal argument gives a subsequence which converges in $L^2(\Omega \times [0,T])$ for any compact subset $\Omega$ of $\R^4$. Note that \Cref{evo_tightness} yields the tightness of $\{|u_k|^2dxdt\}_{k \in \N}$, so it is easy to show the subsequence converges in $L^2(D)$. \par
With the strong convergence of $u_k$ in $L^2(D)$ and the weak convergence criteria in \eqref{nonstationary_convergence}, it is easy to verify the weak limit $(u,p)$ solves $\eqref{navierstokes}$ in the distributional sense.
\end{proof}

\subsection{Parabolic concentration-compactness}

To obtain local energy inequalities for the weak limit $(u,p)$, one would like to pass to the limit $k \rightarrow \infty$ in the local energy inequalities \eqref{evo_local_energy_seq} for the approximation solutions. As we discussed in the introduction, this scenario is critical. In critical variational problems, concentration phenomena may occur. This motivates to look for an analogue of Lions's \cite{lions1985concentration} concentration-compactness principle in parabolic setting. \par
Note that concentration-compactness in elliptic setting may not be applicable to the parabolic setting, since it is hopeless to get $\{\nabla u_k(t)\}_{k \in \N}$ is bounded in $L^2$ for almost every $t$, even for a subsequence. A relevant example in \cite{lopes2001pointwise} by Lopes Filho and Nussenzveig Lopes shows that a bounded sequence in $L^1$ might blow up at almost every point up to any subsequence.

\begin{lemma} \label{parabolic_concentration_compactness}
Given a bounded sequence $\{u_k\}_{k \in \N} \subset L_t^{\infty} L_x^2 \bigcap L_t^2 H_x^1(D)$, let $u$ be given by the limit in \eqref{nonstationary_convergence}. Suppose $u_k$ converges to $u$ in $L^1_{\text{loc}}(D)$. Assume that $\mu_k = |\nabla u_k|^2 dxdt \rightarrow \mu$, $\nu_k = |u_k|^3 dxdt \rightarrow \nu$ weakly in the sense of measures, where $\mu$ and $\nu$ are bounded nonnegative measures on $\R^4 \times [0,T]$. Then there exist nonnegative finite measures $\omega$ and $\lambda$ on $\R^4 \times [0,T]$, such that for any $\varphi \in C_c^{\infty}(D)$,
\begin{align}
  \iint \varphi d\mu &= \iint \varphi |\nabla u|^2 dxdt + \iint \varphi d\lambda, \label{parabolic_concentration_compactness_eq1} \\
  \iint \varphi d\nu &= \iint \varphi |u|^3 dxdt + \iint \varphi d\omega. \label{parabolic_concentration_compactness_eq2}
\end{align}
Moreover, $\omega \ll \lambda$, and we have for any open subdomain $Q$ of $D$,
\begin{equation} \label{parabolic_concentration_compactness_eq0}
  \iint_Q d\omega \leq C \liminf_{k \rightarrow \infty} \|u_k-u\|_{L_t^{\infty}L_x^2(Q)} \iint_Q d\lambda.
\end{equation}
In particular, the Radon–Nikodym derivative satisfies
\begin{equation}
  \frac{d\omega}{d\lambda} 
  \leq C \lim_{r \rightarrow 0} \liminf_{k \rightarrow \infty} \|u_k-u\|_{L_t^{\infty}L_x^2(Q^*_r(x_0,t_0))},
\end{equation}
where $Q^*_r(x_0,t_0) := B_r(x_0) \times (t_0-\frac{r^2}{2}, t_0+\frac{r^2}{2})$.
\end{lemma}

\begin{remark}
We remark that this lemma only requires $\{u_k\}_{k \in \N}$ to be bounded in $L_t^{\infty} L_x^2 \bigcap L_t^2 H_x^1(D)$. $\{u_k\}_{k \in \N}$ does not necessarily solve certain equations. Indeed, we only need $u_k \rightarrow u$ converges in $L^1_{\text{loc}}(D)$.
\end{remark}

\begin{remark}
Although we only state this result for space dimension $4$, one can easily see a trivial generalization to higher dimensions.
\end{remark}

\begin{proof}
Let $v_k = u_k - u \in L_t^{\infty} L_x^2 \bigcap L_t^2 H_x^1$, then
\begin{align}
  v_k \rightarrow 0 \quad &\text{strongly in } L^2_{t,x}, \text{ locally in space}, \\
  v_k \rightarrow 0 \quad &\text{weakly in } L_t^2 \dot{H}_x^1, \\
  v_k \rightarrow 0 \quad &\text{weakly}-* \text{ in } L_t^{\infty} L_x^2.
\end{align}
\par
Define $\omega_k := |v_k|^3 dxdt$. It is easy to check $\{\omega_k\}_{k \in \N}$ is tight. Indeed, for any compact subset $\Omega \subset D$, denote $\Omega^c := (\R^4 \backslash \Omega) \times [0,T]$, then
\[ \|v_k\|_{L^3(\Omega^c)} \leq \|u_k\|_{L^3(\Omega^c)} + \|u\|_{L^3(\Omega^c)}. \]
Because of the weak convergence of $\{\nu_k\}_{k \in \N}$, we know that $\{\nu_k\}_{k \in \N}$ is tight and thus $\|v_k\|_{L^3(\Omega^c)}$ is arbitrarily small given $\Omega$ large enough. Thus we can extract a weakly convergent subsequence with a limit denoted by $\omega$. For any $\varphi \in C_c^{\infty}(D)$, we have
\[
  \begin{split}
  \iint \varphi d\nu &= \lim_{k \rightarrow \infty} \iint \varphi d\nu_k = \lim_{k \rightarrow \infty} \iint \varphi |u_k|^3 dxdt \\
  &= \iint \varphi |u|^3dxdt + \lim_{k \rightarrow \infty} \iint \varphi |v_k|^3dxdt = \iint \varphi |u|^3dxdt + \iint \varphi d \omega.
  \end{split}
\]
The third equality follows from the fact that $u_k \rightarrow u$ in $L^\alpha$ locally in space for $\alpha \in [1,3)$, then all the interaction terms vanish. 
Let $\lambda_k := |\nabla v_k|^2 dxdt \rightarrow \lambda$ weakly in the sense of measures. A similar argument verifies \eqref{parabolic_concentration_compactness_eq2}, and the interaction term vanishes there since $u_k \rightarrow u$ weakly in $L_t^2 \dot{H}_x^1(D)$.
\par
Now we prove \eqref{parabolic_concentration_compactness_eq0}. For any $\varphi \in C_c^{\infty} (D)$, we have
\begin{equation}
  \begin{split}
  \iint_D |\varphi|^3 d\omega &= \lim_{k \rightarrow \infty} \iint_D |\varphi|^3 d\omega_k
    = \lim_{k \rightarrow \infty} \iint_D |v_k\varphi|^3 dxdt \\
    &\leq \liminf_{k \rightarrow \infty} \sup_{0<t<T} \|v_k \varphi\|_{L_x^2} \int_0^T \|v_k \varphi\|^2_{L_x^4} dt \\
    &\leq C \liminf_{k \rightarrow \infty} \sup_{0<t<T} \|v_k \varphi\|_{L_x^2} \iint_D |\nabla (v_k \varphi)|^2 dxdt \\
    &\leq C \liminf_{k \rightarrow \infty} \sup_{0<t<T} \|v_k \varphi\|_{L_x^2} \iint_D |\varphi|^2 |\nabla v_k|^2 dxdt \\
    &\leq C \liminf_{k \rightarrow \infty} \sup_{0<t<T} \|v_k \varphi\|_{L_x^2} \iint_D |\varphi|^2 d\lambda.
  \end{split}
\end{equation}
The first inequality follows from the interpolation between $L^2$ and $L^4$. The second inequality follows from Sobolev embedding. For the third inequality, note that the terms converge to zero when at least one derivative hits $\varphi$. Using smooth functions to approximate the indicator function of $Q$ yields the inequality \eqref{parabolic_concentration_compactness_eq0}.
\par
Therefore, $\omega$ is absolutely continuous with respect to $\lambda$, and by Radon–Nikodym theorem, we have
\[ \frac{d\omega}{d\lambda} \in L^1(D; \lambda) \]
with
\[ \frac{d\omega}{d\lambda} (x_0,t_0) \leq C \lim_{r \rightarrow 0} \liminf_{k \rightarrow \infty} \|v_k\|_{L_t^{\infty}L_x^2(Q^*_r(x_0,t_0))} \]
for any $(x_0,t_0) \in \R^4 \times (0,T)$. 
\end{proof}
\par

Using the parabolic concentration-compactness framework in \Cref{parabolic_concentration_compactness} and the tightness results in \Cref{evo_tightness}, we now can define the notion of weak solution sets involving concentration measures.

\begin{definition} \label{weaksolutionset_nonsta}
The quadruple $(u,p,\lambda,\omega)$ is a \textit{weak solution set} of the Navier-Stokes equations \eqref{navierstokes} if
\begin{enumerate}[leftmargin=*]
    \item $u$ and $p$ are obtained as weak limits of the weak solutions $\{(u_k,p_k)\}_{k \in \N}$ of the regularized Navier-Stokes equations $\eqref{reg_navierstokes}$,  as in \Cref{preparation_existence_nonsta}.
    \item $\lambda$ and $\omega$ are obtained as weak limits of the measures in \Cref{parabolic_concentration_compactness}.
\end{enumerate}
\end{definition}
\par

One can see that every weak solution set comes with a sequence of approximation solutions. However, this is in a sense necessary because a single $L^p$ function is not able to represent concentration of any form. As we shall see, this is effective for analytical purposes in certain critical cases.

\subsection{Local energy inequalities}

In this subsection, we show two energy inequalities with purely local nature for weak solution sets. Although these inequalities are weaker than the local energy inequality \eqref{localenergy} in a sense, they suffice to establish partial regularity of the distributional solutions $(u,p)$. For technical reasons only, in \eqref{evo_local_energy_1} and \eqref{evo_local_energy_2}, we present two distinct forms of these estimates. \par

From the elliptic equation \eqref{pressure_equ} for the pressure $p$, one may guess $p$ has the same regularity as $|u|^2$, so the pressure term in the local energy estimates \eqref{evo_local_energy_seq} may also exhibit concentration of mass. As a preparation for our main goal in this section, we show the concentration in $|up|dxdt$ is localizable and comparable to the concentration in $|u|^3dxdt$.

\begin{lemma} \label{pressure_convergence_nonsta}
Suppose $\{(u_k, p_k)\}_{k \in \N}$ are the solutions of the regularized equations \eqref{reg_navierstokes} and $(u,p,\lambda,\omega)$ is the corresponding weak solution set, then
\[
 \limsup_{k \rightarrow \infty} \iint_D \zeta \big| u_k(p_k-\gamma) - u(p-\gamma) \big| dxdt
    \lesssim \iint_D \zeta d\omega
\]
for any $\zeta \in C_c^{\infty}(D)$ and any $\gamma \in \R$ with $\zeta \geq 0$.
\end{lemma}

\begin{proof}
To prove this result, we need an interpolation inequality. For any $\alpha \in (3,+\infty), \beta \in (2,3), \vartheta \in (0,1)$ with $\frac{1}{\alpha}+\frac{2}{\beta}=1$ and $\frac{1}{\beta} = \frac{\vartheta}{2}+\frac{1-\vartheta}{3}$, we have for any $w \in L_t^{\infty} L_x^2 \cap L_t^2 H_x^1$, 
\begin{equation} \label{pressure_convergence_nonsta_eq1}
  \begin{split}
  \int \|w(t)\|^3_{ L_x^{\beta} } dt &= \int \|w(t)\|_{ L_x^{\beta} } \cdot \|w(t)\|^2_{ L_x^{\beta} } dt \\
    & \leq \|w\|_{ L_t^{\alpha} L_x^{\beta} } \|w\|^2_{ L_{t,x}^{\beta} } \\
    & \leq \|w\|_{ L_t^{\alpha} L_x^{\beta} } \|w\|^{2\vartheta}_{ L_{t,x}^{2} } \|w\|^{2(1-\vartheta)}_{ L_{t,x}^{3} } \\
    & = \|w\|^{2\vartheta}_{ L_{t,x}^{2} } \|w\|^{2(1-\vartheta)}_{ L_{t,x}^{3} } 
      \Big( \int \|w(t)\|_{L_x^{\beta}}^{\alpha} dt \Big)^{1/\alpha} \\
    & \leq \|w\|^{2\vartheta}_{ L_{t,x}^{2} } \|w\|^{2(1-\vartheta)}_{ L_{t,x}^{3} } 
      \Big( \int \|w(t)\|_{L_x^2}^{\alpha(4-\beta)/\beta} \|w(t)\|_{L_x^4}^{\alpha(2\beta-4)/\beta} dt \Big)^{1/\alpha} \\
    & \leq \|w\|^{2\vartheta}_{ L_{t,x}^{2} } \|w\|^{2(1-\vartheta)}_{ L_{t,x}^{3} } 
      \|w\|_{ L_t^{\infty} L_x^2 }^{(4-\beta)/\beta} \|\nabla w\|_{ L_{t,x}^2 }^{2/\alpha}.
  \end{split}
\end{equation}
The first inequality follows from H\"older inequality. The second and the third inequalities follow from Lebesgue interpolation inequality. The fourth ineuqality follows from the Sobolev inequality.
\par
Now we analyze the concentration phenomena of the measures involving the pressure $p_k$. Note the following Poisson equation
\[ - \Delta p_k = \partial_i \partial_j (\tilde{u}_k^iu_k^j), \]
where $\tilde{u}_k := \chi_k * u_k$. From \Cref{rmk_general_ns}, we know this equation holds in the sense of distributions for almost every $t$, then we localize this equation with an arbitrary Lipschitz function $\xi \in C^{0,1}(\R^4)$, i.e.
\[
  \begin{split}
  - \Delta (p_k \xi) =& \xi \partial_i \partial_j (\tilde{u}_k^i u_k^j) - \divr(p_k \nabla \xi) - \nabla p_k \cdot \nabla \xi \\
    =& \partial_i \partial_j (\xi \tilde{u}_k^i u_k^j) - \divr (\tilde{u}_k u_k^j \partial_j \xi) - \partial_j (\tilde{u}_k^i u_k^j) \partial_i \xi - \divr(p_k \nabla \xi) - \nabla p_k \cdot \nabla \xi \\
    =& \partial_i \partial_j (\xi \tilde{u}_k^i u_k^j) - \divr (\tilde{u}_k u_k^j \partial_j \xi + p_k \nabla \xi) - \partial_j (\tilde{u}_k^i u_k^j) \partial_i \xi - \nabla p_k \cdot \nabla \xi.
  \end{split}
\]
\par

Next, we decompose the pressure $p_k\xi=p_k^1+p_k^2+p_k^3$ with
\[
  \begin{split}
  - \Delta p^1_k &= \partial_i \partial_j (\xi \tilde{u}_k^i u_k^j), \\
  - \Delta p^2_k &= - \divr (\tilde{u}_k u_k^j \partial_j \xi + p_k \nabla \xi), \\
  - \Delta p^3_k &= - \partial_j (\tilde{u}_k^i u_k^j) \partial_i \xi - \nabla p_k \cdot \nabla \xi.
  \end{split}
\]
and $p\xi$ in a similar way. Intuitively, the concentration takes place in the component $p^1_k$, since at least one differentiation hits the cutoff function $\xi$ in other components. \par

Now we do rigorous estimates term by term. $p_k^1 \xi$ can be obtained by the Riesz transformation and Calderon-Zygmund theory yields
\begin{equation} \label{CZ_pressure1}
  \begin{split}
  \|p^1_k(t)-p^1(t)\|_{L_x^{3/2}} 
    \lesssim& \|\xi \tilde{u}_k^i(t) u_k^j(t)-\xi \tilde{u}^i(t) u^j(t)\|_{L_x^{3/2}} \\
    \lesssim& \big\|\xi \big(\tilde{u}_k^i(t) - \tilde{u}^i(t)\big) \big(u_k^j(t)- u^j(t)\big)\big\|_{L_x^{3/2}} \\
      &+ \big\|\xi \tilde{u}^i(t) \big(u_k^j(t)- u^j(t)\big)\big\|_{L_x^{3/2}} \\
      &+ \big\|\xi u^j(t) \big(\tilde{u}_k^i(t)- \tilde{u}^i(t)\big)\big\|_{L_x^{3/2}} \\
    \lesssim& \big\|\xi^{1/2} \big(u_k^j(t)- u^j(t)\big)\big\|^2_{L_x^3} \\
      &+ \big\|\xi \tilde{u}^i(t) \big(u_k^j(t)- u^j(t)\big)\big\|_{L_x^{3/2}} \\
      &+ \big\|\xi u^j(t) \big(\tilde{u}_k^i(t)- \tilde{u}^i(t)\big)\big\|_{L_x^{3/2}}. \\
  \end{split}
\end{equation}
Since $\omega_k \rightarrow \omega$ weakly, we have
\[
  \begin{split}
    \limsup_{k \rightarrow \infty} \iint_D |p^1_k-p^1|^{3/2} dxdt
    \lesssim& \limsup_{k \rightarrow \infty} \iint_D |\xi|^{3/2} |u_k-u|^3 dxdt \\
      &+ \limsup_{k \rightarrow \infty} \iint_D |\xi|^{3/2} |\tilde{u}^i(u_k^j- u^j)|^{3/2} dxdt \\
      &+ \limsup_{k \rightarrow \infty} \iint_D |\xi|^{3/2} |u^j(\tilde{u}_k^i- \tilde{u}^i)|^{3/2} dxdt \\
    \lesssim&  \iint_D |\xi|^{3/2} d\omega.
  \end{split}
\]
By Vitali's convergence theorem, the second and third lines converge to zero, because $|u_k^j- u^j|^{3/2}$ is uniformly integrable with respect to $|\xi|^{3/2} |\tilde{u}^i|^{3/2} dxdt$.
\par

Also, for almost every $t$, $u_k(t) \in L^3(\R^4)$, then $p_k^2 \xi$ can be obtained by convolution with singular kernels. Calderon-Zygmund theory yields
\[  
  \begin{split}
  \| p^2_k-p^2 \|_{L^{3/2}_{t,x}(D)} &= \big\|(- \Delta)^{-1} \big[ - \divr\big( (u_k u_k^j-u u^j) \partial_j \xi + (p_k-p) \nabla \xi\big) \big] \big\|_{L^{3/2}_{t,x}(D)} \\
    &\lesssim \big\| (u_k u_k^j-u u^j) \partial_j \xi + (p_k-p) \nabla \xi \big\|_{L^{3/2}_tL^{12/11}_x(D)} \\
    &\lesssim \|\nabla \xi\|_{L^{\infty}} \big(\| u_k u_k^j - uu^j \|_{L^{3/2}_tL^{12/11}_x(D)} + \| p_k-p \|_{L^{3/2}_tL^{12/11}_x(D)} \big) \\
    &\lesssim \|\nabla \xi\|_{L^{\infty}} \big(\| u_k - u \|_{L^3_tL^{24/11}_x(D)} + \| p_k-p \|_{L^{3/2}_tL^{12/11}_x(D)} \big).
  \end{split}
\]
\par
Similarily, for $p_k^3 \xi$ we have
\[  
  \begin{split}
  \| p^3_k-p^3 \|_{L^{3/2}_{t,x}(D)} &\lesssim \|\nabla \xi\|_{L^{\infty}} \big(\| u_k - u \|_{L^3_tL^{24/11}_x(D)} + \| p_k-p \|_{L^{3/2}_tL^{12/11}_x(D)} \big).
  \end{split}
\]

Let $w = u_k-u$ and $\beta = \frac{24}{11}$, the interpolation inequality \eqref{pressure_convergence_nonsta_eq1} yields
\begin{equation} \label{CZ_pressure_eq0} \limsup_{k \rightarrow \infty} \| u_k-u \|_{L_t^3L_x^{24/11}(D)} = 0 \end{equation}
and Calderon-Zygmund theory yields that for $l=1$ or $2$,
\begin{equation} \label{CZ_pressure_eq1} \limsup_{k \rightarrow \infty} \| p_k^l-p^l \|_{L_t^3L_x^{12/11}(D)} = 0. \end{equation}

Now we combine the estimates for $p_k^1$, $p_k^2$ and $p_k^3$. From \eqref{CZ_pressure_eq0} and \eqref{CZ_pressure_eq1}, we know that $p_k^2$ and $p_k^3$ have no contribution to the concentration, then
\begin{equation} \label{CZ_pressure_eq3}
  \begin{split}
  \limsup_{k \rightarrow \infty} \|(p_k-p) \xi\|^{3/2}_{L^{3/2}_{t,x}(D)}
     &\leq\limsup_{k \rightarrow \infty} \sum_{l=1}^3 \|p_k^l-p^l\|^{3/2}_{L^{3/2}_{t,x}(D)} \\
    &\leq \iint_{D_r} |\xi|^{3/2} d\omega.
  \end{split}
\end{equation}
Therefore, we can choose $\xi = \zeta^{2/3}$ and bound the concentration of the measure as follows,
\[ 
  \begin{split}
  \limsup_{k \rightarrow \infty} \iint_{D}& \zeta \big| u_k(p_k-\gamma) - u(p-\gamma) \big|dxdt \\
    \leq& \limsup_{k \rightarrow \infty} \iint_{D} \zeta |u_k||p_k - p| dxdt
      + \limsup_{k \rightarrow \infty} \iint_{D} \zeta |u_k- u||p-\gamma| dxdt \\
    \leq& \limsup_{k \rightarrow \infty} \iint_{D} \zeta |u_k-u||p_k - p| dxdt
      + \limsup_{k \rightarrow \infty} \iint_{D} \zeta |u||p_k - p| dxdt \\
    \lesssim& \iint_{D} \zeta d\omega.
  \end{split}
\]
Due to Vitali's convergence theorem, the second term in the second line and the second term in the third line converge to zero. Note that $\zeta$ is nonnegative and smooth. By \Cref{nonnegative_smooth}, $\xi$ is indeed a compactly supported Lipschitz continuous function. The last inequality follows from \eqref{CZ_pressure_eq3}.
\end{proof}

Now we can prove the following local energy inequalities.

\begin{proposition} \label{energy_convergence}
Let the assumptions be as in \Cref{preparation_existence_nonsta}, then the following local energy inequalities hold,
\begin{equation} \label{evo_local_energy_1}
  \begin{split}
  \limsup_{k \rightarrow \infty}& \sup_t \int_{\R^4} |u_k(t)|^2 \varphi(t) dx + \iint_D |\nabla u|^2 \varphi dxdt + \iint_D \varphi d\lambda \\
  \leq& \iint_D |u|^2 |\partial_t \varphi + \Delta \varphi| dxdt 
    + 2\sum_{i=1}^n \iint_D |u|^3 |\nabla \varphi_i| dxdt +  3 \sum_{i=1}^n \iint_D |\nabla \varphi_i| d\omega \\
    &+ 2\sum_{i=1}^n \iint_D |\nabla \varphi_i| |p-\gamma_i|^{3/2} dxdt + \iint_D f \cdot u \varphi dxdt,
  \end{split}
\end{equation}
\begin{equation} \label{evo_local_energy_2}
  \begin{split}
  \limsup_{k \rightarrow \infty} \sup_t \int_{\R^4} |u_k(t)|^2& \varphi(t) dx + \iint_D |\nabla u|^2 \varphi dxdt + \iint_D \varphi d\lambda \\
  \leq& \iint_D |u|^2 |\partial_t \varphi + \Delta \varphi| dxdt
    + \iint_D |u|^2 (u \cdot \nabla) \varphi dxdt \\
    &+ 2 \iint_D |u-\beta|^3 |\nabla \varphi| dxdt + 3 \iint_D |\nabla \varphi| d\omega \\
    &+ 2\iint_D p (u \cdot \nabla) \varphi dxdt + \iint_D f \cdot u \varphi dxdt,
  \end{split}
\end{equation}
for any $n \in \N$, any functions $\{\gamma_i\}_{1 \leq i \leq n} \subset L^{3/2}([0,T],\R)$ and $\beta \subset L^3([0,T],\R^4)$, any nonnegative cut-off functions $\varphi \in C_c^{\infty}(D)$ with $\varphi(\cdot,0)=0$ and $\{\varphi_i\}_{1 \leq i \leq n} \subset C_c^{\infty}(D)$ with $\varphi = \sum_{i=1}^n \varphi_i$.
\end{proposition}

\begin{proof}
To prove local energy inequalities \eqref{evo_local_energy_1} and \eqref{evo_local_energy_2}, we pass $k$ to infinity in the local energy inequalities for approximation sequence $u_k$. For the cutoff function $\varphi$ defined above, the local energy inequalities \eqref{evo_local_energy_seq} reduces to
\begin{equation} \label{evo_local_energy_seq_tp}
  \begin{split}
  \sup_t \int_{\R^4} &|u_k(t)|^2 \varphi(t)dx + \iint_D |\nabla u_k|^2 \varphi dxdt
  \leq \iint_D |u_k|^2 |\partial_t \varphi + \Delta \varphi| dxdt \\
    &+ \iint_D |u_k|^2 (\tilde{u}_k \cdot  \nabla)\varphi dxdt
      + \iint_D 2 p_k (u_k \cdot \nabla) \varphi dx + \iint_D f \cdot u_k \varphi dxdt.
  \end{split}
\end{equation}
Since $u_k \rightarrow u$ in $L^2(D)$, the convergence of the third and the last terms is straightforward. The convergence of the second term is given by \Cref{evo_tightness} and \eqref{parabolic_concentration_compactness_eq1} in \Cref{parabolic_concentration_compactness}. The difference between the two inequalities and the technical difficulties come from the rest terms, namely the cubic term of $u$ and the term involving $p$. \par

For the cubic term of $u$ in the local energy inequality \eqref{evo_local_energy_1}, note that
\[ \iint_D |\tilde{u}_k|^3 |\nabla \varphi| dxdt = \|h_1+h_2\|_{L^3(D)}^3, \]
where
\[
  \begin{split}
  h_1(x,t) &= \int_{\R^4} u_k(x-y,t) \chi_k(y) \big( |\nabla \varphi(x,t)|^{1/3} - |\nabla \varphi(x-y,t)|^{1/3} \big) dy, \\
  h_2(x,t) &= \int_{\R^4} u_k(x-y,t) \chi_k(y) |\nabla \varphi(x-y,t)|^{1/3} dy. \\
  \end{split}
\]
For $h_1$, note that $d_k:=\diam (\supp \chi_k) \rightarrow 0$ and that $x \rightarrow |x|^{1/3}$ is $1/3-$H{\"o}lder continuous, then Young's inequality for convolution yields
\[
  \begin{split}
  \|h_1\|_{L^3} &\leq 
     \Big\| \int_{\R^4} u_k(x-y,t) \chi_k(y) \frac{|\nabla \varphi(x,t)|^{1/3} 
      - |\nabla \varphi(x-y,t)|^{1/3}}{|y|^{1/3}} d_k^{1/3} dy \Big\|_{L^3(D)} \\
    &\leq C d_k^{1/3} \|\varphi\|_{C^2} \|\tilde{u}_k\|_{L^3(D)} \\
    &\leq C d_k^{1/3} \|\varphi\|_{C^2} \|u_k\|_{L^3(D)}.
  \end{split}
\]
We can then deduce that $h_1$ part converges to zero in $L^3$ when $k$ tends to infinity. For $h_2$, Young’s inequality for convolution yields
\[
  \|h_2\|_{L^3(D)} = \big\| \big(u_k|\nabla \varphi|^{1/3}\big) * \chi_k \big\|_{L^3(D)}
    \leq \big\| u_k |\nabla \varphi|^{1/3} \big\|_{L^3(D)}.
\]
Then these estimates for $h_1$ and $h_2$ yield
\begin{equation} \label{energy_convergence_eq3}
  \begin{split}
  \limsup_{k \rightarrow \infty} \iint_{D} |u_k|^2 (\tilde{u}_k \cdot \nabla) \varphi dxdt
    \leq& \frac{2}{3} \limsup_{k \rightarrow \infty} \iint_{D} |u_k|^3 |\nabla \varphi| dxdt \\
      &+ \frac{1}{3} \limsup_{k \rightarrow \infty} \iint_{D} |\tilde{u}_k|^3 |\nabla \varphi| dxdt \\
    \leq& \iint_{D} |\nabla \varphi| |u|^3dxdt + \iint_{D} |\nabla \varphi| d\omega.
  \end{split}
\end{equation}
\par

For the term involving pressure in the local energy inequality \eqref{evo_local_energy_1},
we use \Cref{pressure_convergence_nonsta} and the fact that $u_k$ is weakly divergence-free to bound
\[ 
  \begin{split}
  \limsup_{k \rightarrow \infty}\iint_{D} p_k (u_k \cdot \nabla) \varphi dxdt
    =& \sum_{i=1}^n \limsup_{k \rightarrow \infty} \iint_{D} p_k u_k \cdot \nabla \varphi_i dxdt \\
    =& \sum_{i=1}^n \limsup_{k \rightarrow \infty} \iint_{D} (p_k-\gamma_i) u_k \cdot \nabla \varphi_i dxdt \\
    \leq& \frac{1}{3} \sum_{i=1}^n \iint_{D} |u|^3 |\nabla \varphi_i| dxdt + \sum_{i=1}^n \iint_{D} |\nabla \varphi_i| d\omega \\
    &+ \sum_{i=1}^n \frac{2}{3} \iint_{D} |\nabla \varphi_i| |p-\gamma_i|^{3/2} dxdt.
  \end{split}
\]
\par

For the cubic term of $u$ in the local energy inequality \eqref{evo_local_energy_2}, we use the fact that $u_k,\tilde{u}_k$ and $u$ are weakly divergence-free. Thus,
\begin{equation} \label{measure_convergence_3_eq1}
  \begin{split}
  \iint_{D} |u_k|^2 (\tilde{u}_k \cdot \nabla) \varphi dxdt
    =& \iint_{D} |u_k - \beta + \beta|^2 (\tilde{u}_k \cdot \nabla) \varphi dxdt \\
    =& \iint_{D} \big[ |u_k-\beta|^2 + 2(u_k-\beta) \cdot \beta \big] 
      (\tilde{u}_k \cdot \nabla) \varphi dxdt \\
    =& \iint_{D} |u_k - \beta|^2 \big[ (\tilde{u}_k-\beta) \cdot \nabla \big] \varphi dxdt \\
      &+ \iint_{D} |u_k - \beta|^2 (\beta \cdot \nabla) \varphi dxdt \\
    &+ 2\iint_{D} \big[(u_k-\beta) \cdot \beta \big] 
      (\tilde{u}_k \cdot \nabla) \varphi dxdt.
  \end{split}
\end{equation}
Next, we argue that the individual terms above can be bounded by the weak limit $u$ and the concentration mass $\omega$. For the term in the third line of \eqref{measure_convergence_3_eq1}, since $\tilde{u}_k-\beta = (u_k - \beta) * \chi_k$, we can apply the same trick by replacing $u_k$ and $\tilde{u}_k$ with $u_k-\beta$ and $\tilde{u}_k-\beta$ and use Young's inequality for convolution, therefore it is sufficient to look at the following term
\[ 
  \begin{split}
  \iint_{D} |u_k - \beta|^2 &\big[ (u_k-\beta) \cdot \nabla \big] \varphi dxdt \\
    \leq& \iint_{D} |u_k - u|^3 |\nabla \varphi| dxdt + \iint_{D} | u-\beta |^3 |\nabla \varphi| dxdt \\
    & + \iint_{D} 3|u_k - u|^2 |u-\beta| |\nabla \varphi| dxdt \\
    & + \iint_{D} 3|u_k - u| |u-\beta|^2 |\nabla \varphi| dxdt \\
    \leq& \iint_{D} |u_k - u|^3 |\nabla \varphi| dxdt + \iint_D | u-\beta |^3 |\nabla \varphi| dxdt \\
    \rightarrow& \iint_{D} |\nabla \varphi| d\omega + \iint_{D} | u-\beta |^3 |\nabla \varphi| dxdt
    \quad \text{as } k \rightarrow \infty.
  \end{split}
\]
Now, we can pass $k \rightarrow \infty$ in the remaining two terms in the last line of \eqref{measure_convergence_3_eq1}. Hence we have
\[ 
  \begin{split}
  \limsup_{k \rightarrow \infty} &\iint_{D} |u_k|^2 (\tilde{u}_k \cdot \nabla) \varphi dxdt
    \leq \iint_{D} |\nabla \varphi| d\omega + \iint_{D} | u-\beta |^3 |\nabla \varphi| dxdt \\
    &+ \iint_{D} |u - \beta|^2 (\beta \cdot \nabla) \varphi dxdt
    + 2 \iint_{D} \big[(u-\beta) \cdot \beta \big] 
      (u \cdot \nabla) \varphi dxdt \\
    =& \iint_{D} |\nabla \varphi| d\omega + \iint_{D} | u-\beta |^3 |\nabla \varphi| dxdt \\
    &+ \iint_{D} |u|^2 (u \cdot \nabla) \varphi dxdt - \iint_{D} |u - \beta|^2 \big[ (u-\beta) \cdot \nabla \big] \varphi dxdt \\
    \leq& \iint_{D} |\nabla \varphi| d\omega + 2\iint_{D} | u-\beta |^3 |\nabla \varphi| dxdt
    + \iint_{D} |u|^2 (u \cdot \nabla) \varphi dxdt.
  \end{split}
\]
\par

Finally, for the term involving pressure in the local energy inequality \eqref{evo_local_energy_2}, \Cref{pressure_convergence_nonsta} yields
\[ 
  \begin{split}
  \limsup_{k \rightarrow \infty}\iint_{D} p_k (u_k \cdot \nabla) \varphi dxdt
    \leq & \iint_{D} p (u \cdot \nabla) \varphi dxdt + \iint_D |\nabla \varphi| d\omega.
  \end{split}
\] 
\end{proof}

\section{Partial regularity theory} \label{partial_regularity_nonstationary}

Partial regularity theory contains deep results of natural scaling and local energy estimates of the Navier-Stokes equations. In this section, we show that weak solution sets have the same scaling invariance as classical solutions, then we adapt Caffarelli, Kohn and Nirenberg's argument to space dimension 4 with the presence of concentration measures. \par

As we mentioned in the introduction, Scheffer proved $\mathcal{H}^3(S) < \infty$. An interesting point is that Scheffer overcame the loss of compactness in $L^3_{t,x}$ by proving uniform local $L^3_{t,x}$ estimate for the approximate solutions $u_k$\footnote{One can see Lemma 2.6 in Scheffer \cite{scheffer1978navier} for details.}, then one can pass the local estimate to the weak limit without splitting the concentration measures and the weak limit $u \in L^3$. In Scheffer's approach, local $L^3_{t,x}$ estimate gives the bound for $\mathcal{H}^3$ measure, while in our work, the $L^2_tH_x^1$ estimate gives more refined bound for $\mathcal{H}^2$ measure. \par

\subsection{Dimensionless estimates in space dimension 4}
The Navier-Stokes equations have a nice scaling property. If $(u,p)$ solves \eqref{navierstokes} with force $f$, then $u_r, p_r$ defined by
\[ u_r(x,t) = r u(rx,r^2t) \quad p_r(x,t) = r^2 p(rx,r^2t) \]
solve \eqref{navierstokes} with force $f_r$ defined by
\[ f_r(x,t) = r^3 f(rx,r^2t). \] \par
The weak solution sets also have a similar scaling property.
\begin{lemma}
If $(u,p,\lambda,\omega)$ is a weak solution set of the Navier-Stokes equations \eqref{navierstokes} with external force $f$, then for any $r>0$, the scaled quadruple $(u_r,p_r,\lambda_r,\omega_r)$ is also a weak solution set of \eqref{navierstokes} with external force $f_r$, where $u_r,p_r$ and $f_r$ are defined as above and $\lambda_r,\omega_r$ are defined as
\[ 
  \begin{split}
  \iint_E d\lambda_r &:= r^{-2}\iint_{\{(rx,r^2t)|(x,t) \in E\}} d\lambda \\
  \iint_E d\omega_r  &:= r^{-3} \iint_{\{(rx,r^2t)|(x,t) \in E\}} d\omega
  \end{split}
\]
for any $E \subset \R^4 \times \R$. 
\end{lemma}

\par

For a weak solution set $(u,p,\lambda,\omega)$, we give short-hand notations for the following scale-invariant quantities.
\begin{equation}
  \begin{split}
  A(x_0, t_0, r) &= \limsup_{k \rightarrow +\infty} \underset{t_0-r^2<t<t_0}{\sup} r^{-2} \int_{B_r(x_0)} |u_k(t)|^2 dx \\
  \delta (x_0, t_0, r) &=  r^{-2} \iint_{Q_r(x_0,t_0)} |\nabla u|^2 dx dt \\
  \delta_c (x_0, t_0, r) &=  r^{-2} \iint_{Q_r(x_0,t_0)} d\lambda \\
  G(x_0, t_0, r) &= r^{-3} \iint_{Q_r(x_0,t_0)} |u|^3 dx dt \\
  G_c(x_0, t_0, r) &= r^{-3} \iint_{Q_r(x_0,t_0)} d\omega \\
  H(x_0, t_0, r) &= r^{-3} \iint_{Q_r(x_0,t_0)} |u-\tilde{u}_{r,x_0}|^3 dx dt \\
  K(x_0, t_0, r) &= r^{-3} \iint_{Q_r(x_0,t_0)} |p|^{3/2} dx dt \\
  L(x_0, t_0, r) &= r^{-3} \iint_{Q_r(x_0,t_0)} |p-\tilde{p}_{r,x_0}|^{3/2} dx dt \\
  F_1(x_0, t_0, r) &= r^{3q-6} \iint_{Q_r(x_0,t_0)} |f|^q dx dt \\
  F_2(x_0, t_0, r) &= \iint_{Q_r(x_0,t_0)} |f|^2 dx dt
  \end{split}
\end{equation}
where
\[
  \begin{split}
  \tilde{u}_{r,x_0}(t) &= \frac{1}{\mathcal{L}(B_r)} \int_{B_r(x_0)} u(x,t) dx \\
  \tilde{p}_{r,x_0}(t) &= \frac{1}{\mathcal{L}(B_r)} \int_{B_r(x_0)} p(x,t) dx
  \end{split} \]
and $Q_r(x,t)$ is the parabolic cylinder centered at $(x,t)$ given by
\[ Q_r(x,t) := B_r(x) \times ( t-r^2, t ). \]
When $(x_0,t_0)=(0,0)$, we abbreviate $A(0, 0, r)$ to $A(r)$. This convention also applies to other quantities and parabolic cylinders. For technical reasons, we also need another quantity $L'$ which is not scale-invariant.
\[ L'(x_0, t_0, r) = r^{-{5/2}} \iint_{Q_r(x_0,t_0)} |p-\tilde{p}_{r,x_0}|^{3/2} dx dt 
    = r^{1/2}L(x_0, t_0, r) \]
Note that already in the work \cite{caffarelli1982partial} of Caffarelli, Kohn and Nirenberg, a quantity similar to $L'$ that is not scale-invariant plays an important role.
\par
A crucial component of proving partial regularity in space dimension 4 is interpolation inequalities. Next we introduce three interpolation inequalities based on the above dimensionless quantities.
\begin{lemma} \label{interpolation}
Suppose that $(u,p,\lambda,\omega)$ is a weak solution set of the Navier-Stokes equations \eqref{navierstokes} in space dimension $4$ in $Q_r(x_0,t_0)$. Then there exists an absolute constant $C_1>0$, which is independent of $(x_0,t_0) \in \R^4 \times \R$ and $r>0$, such that
\[
  \begin{split}
  G(x_0,t_0,r) &\leq C_1 A^{3/2}(x_0,t_0,r) + C_1 \delta^{3/2}(x_0,t_0,r), \\
  G_c(x_0,t_0,r) &\leq C_1 A^{1/2}(x_0,t_0,r) \delta_c(x_0,t_0,r), \\
  H(x_0,t_0,r) &\leq C_1 A^{1/2}(x_0,t_0,r) \delta(x_0,t_0,r).
  \end{split}
\]
\end{lemma}

\begin{proof}
Since all quantities here are scale-invariant, it suffices to prove these inequalities for $r=1$.
By Lebesgue interpolation inequality,
\[  
  \begin{split}
  \|u\|_{L^3(B_1(x_0))} 
    &\leq \|u\|_{L^4(B_1(x_0))}^{2/3} \|u\|_{L^2(B_1(x_0))}^{1/3}, \\
  \|u-\tilde{u}_{1,x_0}\|_{L^3(B_1(x_0))} 
    &\leq \|u-\tilde{u}_{1,x_0}\|_{L^4(B_1(x_0))}^{2/3} \|u-\tilde{u}_{1,x_0}\|_{L^2(B_1(x_0))}^{1/3} \\
    &\leq \|u-\tilde{u}_{1,x_0}\|_{L^4(B_1(x_0))}^{2/3} \|u\|_{L^2(B_1(x_0))}^{1/3}.
  \end{split}
\]
By Sobolev embedding and Sobolev-Poincar\'{e} inequality,
\[
  \begin{split}
  \|u\|_{L^4(B_1(x_0))} &\lesssim \|u\|_{L^2(B_1(x_0))} + \|\nabla u\|_{L^2(B_1(x_0))} \\
  \|u-\tilde{u}_{1,x_0}\|_{L^4(B_1(x_0))} &\lesssim  \|\nabla u\|_{L^2(B_1(x_0))}.
  \end{split}
\]
Then we integrate in time and use Young's inequality,
\[ 
  \begin{split}
  \iint_{Q_1(x_0,t_0)} |u|^3 dx dt
  &\lesssim \int_{t_0-1}^{t_0} \Big( A(x_0,t_0,1) + \int_{B_1(x_0)} |\nabla u|^2 dx \Big) 
    A^{1/2}(x_0,t_0,1) dt \\
  &= A^{3/2}(x_0,t_0,1) + A^{1/2}(x_0,t_0,1) \delta(x_0,t_0,1) \\
  &\lesssim A^{3/2}(x_0,t_0,1) + \delta^{3/2}(x_0,t_0,1).
  \end{split}
\]
In the first inequality, we use lower semi-continuity of the weak-$*$ convergence to bound $\|u\|_{L_t^{\infty}L_x^2}$ with $\limsup \|u_k\|_{L_t^{\infty}L_x^2}$. Similarly, we also have
\[
  \iint_{Q_1(x_0,t_0)} |u-\tilde{u}_{r,x_0}|^3 dxdt \lesssim A^{1/2}(x_0,t_0,1) \delta(x_0,t_0,1).
\]
\par
The second interpolation inequality follows directly from \Cref{parabolic_concentration_compactness}.
\par
As we mentioned, these quantities are scale-invariant. Then we can obtain the inequalities for $r \neq 1$ via scaling.
\end{proof}

The second key ingredient is the local energy inequalities \eqref{evo_local_energy_1} and \eqref{evo_local_energy_2}. To use these local energy inequalities, we also need different types of estimates for the pressure term. We prove a 4-dimensional analogue of Lemma 3.2 in \cite{caffarelli1982partial}.

\begin{lemma} \label{pressurepreestimate}
Suppose that $(u,p,\lambda,\omega)$ is a weak solution set of the Navier-Stokes equations \eqref{navierstokes} in space dimension $4$ in $Q_{\rho}(x_0,t_0)$. Then there exists an absolute constant $C_2>0$, which is independent of $(x_0,t_0) \in \R^4 \times \R$ and $\rho>0$, such that
\begin{equation} \label{pressurepreestimate_eq0}
  \begin{split}
  L'(x_0, t_0, r) \leq & C_2 r^{-5/2} \iint_{Q_{2r}(x_0,t_0)} |u|^3 dxdt 
    + C_2 r^5 \Big( \underset{t_0-r^2<t<t_0}{\sup} \int_{2r<|y-x_0|<\rho} \frac{|u|^2}{|y-x_0|^5} dy \Big)^{3/2} \\
  + & C_2 \frac{r^3}{\rho^{11/2}} \iint_{Q_{\rho}(x_0,t_0)} \big(|u|^3 + |p|^{3/2}\big) dxdt,
  \end{split}
\end{equation}
where $0<r\leq\frac{\rho}{2}$.
\end{lemma}

\begin{proof}[Proof of \Cref{pressurepreestimate}]
It suffices to prove the estimate when $(x_0,t_0)=(0,0)$. Choose a cutoff function $\psi \in C_c^{\infty}(\R^4)$ such that $0 \leq \psi \leq 1$ and
\begin{equation} \label{pressurepreestimate_eq1}
  \begin{split}
  \psi \equiv 1 \text{ in } B_{3\rho/4}, \quad
  \psi \equiv 0 \text{ in } \R^4 \backslash B_{\rho}, \quad
  |\nabla \psi| \lesssim \rho^{-1}, \quad
  |\nabla^2 \psi| \lesssim \rho^{-2}.
  \end{split}
\end{equation}
Then we localize the pressure equation and integrate by parts to move the differentiation from $u$ and $p$ to $\psi$,
\begin{equation} \label{pressurepreestimate_not1}
  \begin{split}
  p(x,t)\psi(x) =& (-\Delta)^{-1} (-\Delta) \big(p(x,t)\psi(x)\big) \\
  =& \frac{1}{4\pi^2} \int_{\R^4} \frac{1}{|x-y|^2}
    \big( \psi \partial_i \partial_j(u_iu_j) -2\nabla \psi \cdot \nabla p - p\Delta \psi \big) dy \\
  =& \frac{1}{4\pi^2} \int_{\R^4} u_iu_j \psi \partial_i \partial_j \Big( \frac{1}{|x-y|^2} \Big) dy \\
  +& \frac{1}{4\pi^2} \int_{\R^4} u_iu_j \Big( \frac{\partial_i \partial_j \psi}{|x-y|^2}
                        + \partial_j \psi \frac{4(x_i-y_i)}{|x-y|^4} \Big) dy \\
  +& \frac{1}{4\pi^2} \int_{\R^4} p \Big( \frac{\Delta \psi}{|x-y|^2} 
                      + \frac{4(x-y) \cdot \nabla \psi}{|x-y|^4} \Big) dy \\
  =& p_1(x,t) + p_2(x,t) + p_3(x,t) + p_4(x,t),
  \end{split}
\end{equation}
where by the fact that $\psi$ is supported in $B_{\rho}$,
\begin{equation} \label{pressurepreestimate_not2}
  \begin{split}
  p_1(x,t) &= \frac{1}{4\pi^2} \int_{B_{2r}} u_iu_j \psi \partial_i \partial_j \Big( \frac{1}{|x-y|^2} \Big) dy, \\
  p_2(x,t) &= \frac{1}{4\pi^2} \int_{B_{\rho} \backslash B_{2r}} u_iu_j \psi \partial_i \partial_j \Big( \frac{1}{|x-y|^2} \Big) dy, \\
  p_3(x,t) &= \frac{1}{4\pi^2} \int_{B_{\rho}} u_iu_j \Big( \frac{\partial_i \partial_j \psi}{|x-y|^2}
                        + \partial_j \psi \frac{4(x_i-y_i)}{|x-y|^4} \Big) dy, \\
  p_4(x,t) &= \frac{1}{4\pi^2} \int_{B_{\rho}} p \Big( \frac{\Delta \psi}{|x-y|^2} 
                      + \frac{4(x-y) \cdot \nabla \psi}{|x-y|^4} \Big) dy.
  \end{split}
\end{equation}
Now, we decompose $L'(x_0, t_0, r)$ into four terms involving $p_1,p_2,p_3$ and $p_4$ respectively and estimate them separately,
\begin{equation} \label{pressurepreestimate_eq2}
  L'(x_0, t_0, r) \leq \sum_{l=1}^4 r^{-5/2} \iint_{Q_r} |p_l-\tilde{p}_{l,r}|^{3/2} dxdt.
\end{equation} \par
We interpret $p_1$ as $p_1 = T_{ij} \big( u_iu_j \psi \big)$, where singular integral operators $\{T_{ij}\}_{1 \leq i,j \leq 4}$ are given by
\begin{equation} \label{Toperator}
  T_{ij} \zeta = \Big( \partial_i \partial_j \Big( \frac{1}{|x|^2} \Big) \Big) \ast \zeta.
\end{equation}
From Calder\'on-Zygmund theory we know $\{T_{ij}\}_{1 \leq i,j \leq 4}$ are bounded linear operators from $L^q(\R^4)$ to $L^q(\R^4)$ for any $1 < q < \infty$, hence let 
\[ \zeta(y,t) = u_i(y,t)u_j(y,t) \psi(y) \boldsymbol{1}_{\{y \in B_{2r}\}} \] and it yields
\[ \int_{B_r} |p_1|^{3/2} dx \lesssim \int_{B_{2r}} |u|^3 dx. \]
By a simple computation and integrating in time, we have
\begin{equation} \label{pressurepreestimate_eq3}
  \iint_{Q_r} |p_1-\tilde{p}_{1,r}|^{3/2} dxdt \leq C_2 \iint_{Q_{2r}} |u|^3 dxdt.
\end{equation}
\par
We estimate the remaining terms by bounding the $L^{\infty}$-norm of the space derivatives of the pressure $p$.
\par
For $p_2$, we can control its derivative as follows. When $(x,t) \in Q_r$,
\[ |\nabla p_2(x,t)| \lesssim \int_{2r<|y|<\rho} \frac{\psi|u|^2}{|x-y|^5} dx
          \lesssim \int_{2r<|y|<\rho} \frac{|u|^2}{|y|^5} dx. \]
The second inequality follows from $2|x-y|>|y|$ when $x \in B_r, y \in B^c_{2r}$. Then we can estimate the second term in \eqref{pressurepreestimate_eq2} by mean value theorem as follows,
\begin{equation}  \label{pressurepreestimate_eq4}
  \begin{split}
  \iint_{Q_r} |p_2-\tilde{p}_{2,r}|^{3/2} dxdt 
    &\leq \frac{\pi^2}{2} r^4\int_{-r^2}^0 \|p_2-\tilde{p}_{2,r}\|^{3/2}_{L^{\infty}(B_r)} dt \\
    &\lesssim \frac{\pi^2}{2}r^{11/2} \int_{-r^2}^0 \|\nabla p_2\|^{3/2}_{L^{\infty}(B_r)} dt \\
    &\lesssim r^{15/2} \Big( \underset{-r^2<t<0}{\sup} \int_{2r<|y|<\rho} \frac{|u|^2}{|y|^5} dx \Big)^{3/2}.
  \end{split}
\end{equation}
Similarly, for $p_3$ and $p_4$, note that $\nabla \psi =0$ and $\nabla^2 \psi =0$ in $B_{3 \rho/4}$. Moreover, when $x \in B_r$ and $y \in B_{\rho} \backslash B_{3\rho/4}$, $|x-y|>\frac{\rho}{4}$. Hence for  $(x,t) \in Q_r$,
\begin{equation} \label{pressurepreestimate_eq5}
  \begin{split}
  |\nabla p_3(x,t)| &\lesssim \int_{B_{\rho} \backslash B_{3\rho/4}} |u|^2
                \Big( \frac{|\nabla^2 \psi|}{|x-y|^3} + \frac{|\nabla \psi|}{|x-y|^4} \Big) dy
           \lesssim \rho^{-5} \int_{B_{\rho} \backslash B_{3\rho/4}} |u|^2 dy, \\
  |\nabla p_4(x,t)| &\lesssim \int_{B_{\rho} \backslash B_{3\rho/4}} |p|
                \Big( \frac{|\nabla^2 \psi|}{|x-y|^3} + \frac{|\nabla \psi|}{|x-y|^4} \Big) dy
           \lesssim \rho^{-5} \int_{B_{\rho} \backslash B_{3\rho/4}} |p| dy.
  \end{split}
\end{equation}
Thus,
\begin{equation}  \label{pressurepreestimate_eq6}
  \begin{split}
  \sum_{l=3}^4 \iint_{Q_r} |p_l-\tilde{p}_{l,r}|^{3/2} dxdt
    &\leq \sum_{l=3}^4 \frac{\pi^2}{2} r^4\int_{-r^2}^0 \|p_l-\tilde{p}_{l,r}\|^{3/2}_{L^{\infty}(B_r)} dt \\
    &\lesssim \sum_{l=3}^4 \frac{\pi^2}{2} r^{11/2}\int_{-r^2}^0 \|\nabla p_l\|^{3/2}_{L^{\infty}(B_r)} dt \\
    &\lesssim \Big( \frac{r}{\rho}\Big)^{11/2} \int_{Q_{\rho}} |u|^3+|p|^{3/2} dxdt.
  \end{split}
\end{equation}
The second inequality follows from mean value theorem and the last one follows from \eqref{pressurepreestimate_eq5} and H{\"o}lder's inequality. \par
Finally, combining the estimates \eqref{pressurepreestimate_eq3}, \eqref{pressurepreestimate_eq4} and \eqref{pressurepreestimate_eq6} and dividing them by $r^{5/2}$ yield \eqref{pressurepreestimate_eq0}.
\end{proof}

To obtain partial regularity theory in space dimension 4, we also need another estimate for the pressure $p$.

\begin{lemma} \label{pressurepreestimate_2}
Suppose that $(u,p,\lambda,\omega)$ is a weak solution set of the Navier-Stokes equations \eqref{navierstokes} in space dimension $4$ in $Q_r(x_0,t_0)$. Then there exists an absolute constant $C_3>0$, which is independent of $(x_0,t_0) \in \R^4 \times \R$ and $r>0$, such that
\begin{equation} \label{pressurepreestimate_2_eq0}
  \begin{split}
  K(x_0, t_0, \theta r) \leq & C_1C_3 \theta^{-3} A^{1/2}(x_0, t_0, r) \delta(x_0, t_0, r) 
    + C_3 \theta K(x_0, t_0, r)
  \end{split}
\end{equation}
for any $\theta \in (0,\frac{1}{2}]$. The constant $C_1>0$ is absolute and comes from \Cref{interpolation}.
\end{lemma}

\begin{proof}
Again it suffices to prove the estimate for $(x_0,t_0)=(0,0)$. Choose a cutoff function $\psi \in C_c^{\infty}(\R^4)$ such that $0 \leq \psi \leq 1$ and
\begin{equation} \label{pressurepreestimate_2_eq1}
  \begin{split}
  \psi \equiv 1 \text{ in } B_{3r/4}, \quad
  \psi \equiv 0 \text{ in }  \R^4 \backslash B_r, \quad
  |\nabla \psi| \lesssim r^{-1}, \quad
  |\nabla^2 \psi| \lesssim r^{-2}.
  \end{split}
\end{equation}
The pressure equation can be written as
\[ -\Delta p = \partial_i \partial_j \big[ (u_i-\tilde{u}_{i,r})(u_j-\tilde{u}_{j,r}) \big]. \]
We can localize this equation like \eqref{pressurepreestimate_not1} and \eqref{pressurepreestimate_not2} to obtain
\[
  \begin{split}
  p(x,t)\psi(x) &= p_1(x,t) + p_2(x,t) + p_3(x,t), \\
  p_1(x,t) &= \frac{1}{4\pi^2} \int_{B_r} (u_i-\tilde{u}_{i,r})(u_j-\tilde{u}_{j,r}) \psi 
    \partial_i \partial_j \Big( \frac{1}{|x-y|^2} \Big) dy, \\
  p_2(x,t) &= \frac{1}{4\pi^2} \int_{B_r} (u_i-\tilde{u}_{i,r})(u_j-\tilde{u}_{j,r}) 
    \Big( \frac{\partial_i \partial_j \psi}{|x-y|^2} + \partial_j \psi \frac{4(x_i-y_i)}{|x-y|^4} \Big) dy, \\
  p_3(x,t) &= \frac{1}{4\pi^2} \int_{B_r} p \Big( \frac{\Delta \psi}{|x-y|^2} 
                      + \frac{4(x-y) \cdot \nabla \psi}{|x-y|^4} \Big) dy.
  \end{split}
\] \par
For $p_1$, Calder\'on-Zygmund theory yields
\[ \int_{B_{\theta r}} |p_1|^{3/2} dx \lesssim \int_{B_r} |(u_i-\tilde{u}_{i,r})(u_j-\tilde{u}_{j,r})|^{3/2} dx 
  \lesssim \int_{B_r} |u-\tilde{u}_r|^3 dx. \]
Integrating in time gives
\[ \iint_{Q_{\theta r}} |p_1|^{3/2} dxdt \lesssim \iint_{Q_r} |u-\tilde{u}_r|^3 dxdt. \] \par
For $p_2$, note that $\nabla \psi$ is supported in $B_r \backslash B_{3r/4}$. Then for $x \in B_{\theta r}$, $|x-y|>\frac{r}{4}$ and the bounds in \eqref{pressurepreestimate_2_eq1} give
\[ |p_2| \lesssim r^{-4} \int_{B_r} |u-\tilde{u}_r|^2 dx. \]
Then integrate in $Q_{\theta r}$ to obtain
\[ \iint_{Q_{\theta r}} |p_2|^{3/2} dxdt 
  \lesssim \int_{-(\theta r)^2}^0 (\theta r)^4 \|p_2\|^{3/2}_{L^{\infty}(B_{\theta r})} dt
  \lesssim \theta^4 \iint_{Q_r} |u-\tilde{u}_r|^3 dxdt. \] \par
For $p_3$, likewise, we have
\[ \iint_{Q_{\theta r}} |p_3|^{3/2} dxdt \lesssim \theta^4 \iint_{Q_r} |p|^{3/2} dxdt. \] \par
Combining the estimates for $p_1,p_2$ and $p_3$ and applying the interpolation inequality \Cref{interpolation}, we have
\[ K(\theta r) \lesssim C_1 \theta^{-3} A^{1/2}(r) \delta(r) + C_3 \theta K(r), \]
as claimed.
\end{proof}

Two types of cutoff functions are introduced in following lemmas, respectively for two local partial regularity results which we will show later. Similar cutoff functions have been used by Scheffer \cite{scheffer1980navier} and Caffarelli, Kohn, and Nirenberg \cite{caffarelli1982partial}.
\begin{lemma} \label{backwardheat}
Let $r_n = 2^{-n}$ and $Q_n = Q_{r_n}$. In space dimension $4$, $\{ \phi_n \}_{n \in \N}$ is a sequence of localized solutions of backward heat equations given by
\[ \phi_n(x,t) = \chi(x,t) \Psi_n(x,t) \quad (x,t) \in \R^4 \times (-\infty, 0), \]
where $\{\Psi_n\}_{n \in \N}$ are the solutions of backward heat equations given by
\[ \Psi_n(x,t) = \frac{1}{(r_n^2-t)^2} \exp \Big( -\frac{|x|^2}{4(r_n^2-t)} \Big) \]
and $\chi$ is a cut-off function such that
\[
  \begin{split}
  \chi \equiv 1 \text{ in } Q_{1/4}, \quad
  \chi \equiv 0 \text{ in } \R^4 \times (-\infty, 0) \backslash Q_{1/3},
  \end{split}
\]
then the following statements hold for any integer $n \in \N$:

\begin{enumerate}
  \item $\partial_t\phi_n + \Delta \phi_n = 0$ in $Q_{1/4}$;
  \item $|\partial_t\phi_n + \Delta \phi_n| \leq C_4$ in $\R^4 \times (-\infty, 0)$;
  \item $C_4^{-1}r^{-4}_n \leq \phi_n \leq C_4 r^{-4}_n$ and $|\nabla \phi_n| \leq C_4r_n^{-5}$ in $Q_n$;
  \item $ \phi_n \leq C_4r_k^{-4} $ and $ | \nabla \phi_n | \leq C_4r_k^{-5} $ in $Q_{k-1} \backslash Q_k$ for any $2 \leq k \leq n$.
\end{enumerate}
Note that the constant $C_4>0$ is absolute.
\end{lemma}
\begin{proof}
The first statement is obvious. For the second we compute
\[
  \begin{split}
  \partial_t\phi_n + \Delta \phi_n 
  &= \Psi_n(\partial_t\chi + \Delta \chi) + \chi(\partial_t\Psi_n + \Delta \Psi_n) 
    + 2 \nabla \Psi_n \cdot \nabla \chi \\
  &= \Psi_n(\partial_t\chi + \Delta \chi) + 2 \nabla \Psi_n \cdot \nabla \chi.
  \end{split}
\]
Because any derivative of $\chi$ vanishes in $Q_{1/4}$ and $\Psi_n, \nabla \Psi_n$ are uniformly bounded in $(x,t) \in \R^4 \times (-\infty, 0) \backslash Q_{1/4}$, we can deduce that $|\partial_t\phi_n + \Delta \phi_n|$ is bounded uniformly in $(x,t) \in \R^4 \times (-\infty, 0)$ and in $n \in \N$. \par
For the third, if $(x,t) \in Q_n$, then $r_n^2 \leq r_n^2-t \leq 2r_n^2$ and $|x|^2 \leq r_n^2$. We compute
\[  \nabla \phi_n(x,t) = 
    \Big( \frac{\nabla \chi(x,t)}{(r_n^2-t)^2} - \frac{x \chi(x,t)}{2(r_n^2-t)^3} \Big)
    \exp \Big( -\frac{|x|^2}{4(r_n^2-t)} \Big).  \]
The terms $\chi,\nabla \chi$ and $\exp \Big( -\frac{|x|^2}{4(r_n^2-t)} \Big)$ are bounded from above and from below uniformly in $(x,t) \in Q_n$ and in $n \in \N$. Then the third statement follows from
\[ \frac{1}{(r_n^2-t)^{2}} \leq r_n^{-4} \quad \frac{|x|}{(r_n^2-t)^3} \leq r_n^{-5}. \]
\par
For the fourth, if $(x,t) \in Q_{k-1} \backslash Q_k$ and $t \leq -r_k^2$, we have $|x|^2 \leq r_{k-1}^2$ and $r_n^2 - t \geq r_n^2+r_k^2$, then this statement follows from the argument for the third one. If $(x,t) \in Q_{k-1} \backslash Q_k$ and $t > -r_k^2$, then $r_k^2 \leq |x|^2 \leq r_{k-1}^2$ and $r_n^2 \leq r_n^2-t \leq r_n^2+r_k^2$, thus
\[ \phi_n(x,t) \leq \frac{\chi}{(r_n^2-t)^2} \exp \Big( -\frac{r_k^2}{4(r_n^2-t)} \Big) \leq \chi r_k^{-4} \alpha^4 e^{-\alpha^2/4}, \]
where $\alpha = r_k(r_n^2-t)^{-1/2}$ and the function $\alpha^4 e^{-\alpha^2/4}$ is uniformly bounded. The bound for $\nabla \phi_n$ follows similarly.
\end{proof}

\begin{lemma} \label{alter_backwardheat}
In space dimension $4$, fix $r>0$. For any $0 < \theta \leq \frac{1}{2}$ we define
\[ \phi_{\theta}(x,t) = \frac{1}{[(\theta r)^2-t]^2} 
  \exp\Big( -\frac{|x|^2}{4[(\theta r)^2-t]} \Big) \chi \Big( \frac{x}{r},\frac{t}{r^2} \Big)
  \quad (x,t) \in \R^4 \times (-\infty, 0), \]
where $\chi \in C^{\infty}_c(B_1 \times (-1,1))$ is a cutoff function such that $\chi \equiv 1$ in $B_{1/2} \times (-\frac{1}{4},\frac{1}{4})$. Then there exists an absolute constant $C_5>0$ such that
\begin{enumerate}
  \item $C_5^{-1} (\theta r)^{-4} \leq \phi_{\theta} \leq C_5 (\theta r)^{-4}$ in $Q_{\theta r}$;
  \item In $Q_r$, we have following bounds,
  \[ 
    \begin{split}
    \phi_{\theta} &\leq C_5 (\theta r)^{-4}, \\
    |\nabla \phi_{\theta}| &\leq C_5 (\theta r)^{-5}, \\
    |\partial_t \phi_{\theta} + \Delta \phi_{\theta}| &\leq C_5 r^{-6}.
    \end{split}
  \]
\end{enumerate}
\end{lemma}
\begin{proof}
This proof is analogue to the proof of \Cref{backwardheat}.
\end{proof}

These estimates will be fundamental for the local partial regularity results of the Navier-Stokes equations in space dimension 4. They involve some constants $C_1,C_2,C_3,C_4$ and $C_5$. All of these constants are absolute.

\subsection{Partial regularity results}

The first partial regularity result states that $u$ is locally bounded if $u,p,f$ and concentration measure $\omega$ satisfy a local smallness condition. This result is a version of Proposition 1 in \cite{caffarelli1982partial} in space dimension 4 with concentration measures.

\begin{proposition} \label{cknprop}
There exist an absolute constant $\varepsilon >0$ and, for any fixed $q>3$, constants $\kappa = \kappa(\varepsilon,q)$ and $C=C(\varepsilon,q)$ depending on $\varepsilon$ and $q$ with the following property. If a weak solution set $(u,p,\lambda,\omega)$ of the Navier-Stokes equations \eqref{navierstokes} in $Q_1(0,0)$ in space dimension $4$ satisfies
\begin{equation} \label{smallness}
  \begin{split}
  \iint_{Q_1} \Big( |u|^3 + |p|^{3/2} \Big) dxdt + \iint_{Q_1} d\omega &\leq \varepsilon \\
  \iint_{Q_1}  |f|^{q} dxdt &\leq \kappa,
  \end{split}
\end{equation}
then $ \|u\|_{ L^{\infty}(Q_{1/2}(0,0)) } < C $.
\end{proposition}

\begin{proof}
Let $r_n = 2^{-n}$ and $Q_n = Q_{r_n}$, $n \geq 2$. The strategy is to iteratively prove the following estimates
\begin{equation} \label{ainduction}
G(x_0,t_0,r_n) + G_c(x_0,t_0,r_n) + L'(x_0,t_0,r_n) \leq \varepsilon^{2/3} r_n^3
\end{equation}
\begin{equation} \label{binduction}
A(x_0,t_0,r_n) + \delta (x_0,t_0,r_n) + \delta_c (x_0,t_0,r_n) \leq C_B \varepsilon^{2/3} r_n^2
\end{equation}
for all $n \in \N$. We use $\sum^n_{k=1} A(x_0,t_0,r_k)$, $\sum^n_{k=1} \delta(x_0,t_0,r_k)$ and $\sum^n_{k=1} \delta_c(x_0,t_0,r_k)$ to control $G(x_0,t_0,r_{n+1})$, $G_c(x_0,t_0,r_{n+1})$ and $L'(x_0,t_0,r_{n+1})$ by means of the interpolation inequalities in \Cref{interpolation} and the estimate for the pressure $p$ in \Cref{pressurepreestimate}. Conversely, we bound $A(x_0,t_0,r_{n+1})$, $\delta(x_0,t_0,r_{n+1})$ and $\delta_c(x_0,t_0,r_{n+1})$ through $\sum^n_{k=1} G(x_0,t_0,r_k)$, $\sum^n_{k=1} G_c(x_0,t_0,r_k)$ and $\sum^n_{k=1} L'(x_0,t_0,r_k)$, by means of the local energy inequality \eqref{evo_local_energy_1}. 

For any $(x_0,t_0) \in Q_{1/2}(0,0)$, we will prove \eqref{ainduction} and \eqref{binduction} inductively. In the rest of this proof, we use $\eqref{ainduction}_k$ to denote the inequality \eqref{ainduction} with index $k \in \N$. This notation also applies to \eqref{binduction}.
\vspace{2mm}\\
\noindent 
\textbf{Claim 1:} \textit{The inequality $\eqref{ainduction}_1$ holds.} \par
\begin{proof}[Proof of Claim 1]
H{\"o}lder's inequality gives
\[ 
  \begin{split}
  G(x_0&, t_0, r_1) + G_c(x_0, t_0, r_1) + L(x_0, t_0, r_1) \\
  &\leq 8 \iint_{Q_{1/2}(x_0,t_0)} |u|^3 dxdt + 8 \iint_{Q_{1/2}(x_0,t_0)} d\omega +
    16 \iint_{Q_{1/2}(x_0,t_0)} |p|^{3/2} dxdt.
  \end{split}
\]
Then we impose the first condition on $\varepsilon>0$,
\begin{equation} \label{choseepsilon1}
  \varepsilon \leq 2^{-21}
\end{equation}
Now we can invoke initial smallness condition and it yields
\[ G(x_0, t_0, r_1) + G_c(x_0, t_0, r_1) + L(x_0, t_0, r_1) \leq 16 \varepsilon \leq \varepsilon^{2/3} r_1^3. \]
\end{proof}
\vspace{2mm}
\noindent 
\textbf{Claim 2:} \textit{$\{\eqref{ainduction}_k\}_{1 \leq k \leq n}$ implies $\eqref{binduction}_{n+1}$.} \par

\begin{proof}[Proof of Claim 2]

Let $\phi_n$ be the localized solution of the backward heat equation
\[ \phi_n(x,t) = \frac{\chi(x,t)}{(r_n^2-t)^{2}} \exp \Big( -\frac{|x|^2}{4(r_n^2-t)} \Big) \]
with $\chi$ as given in \Cref{backwardheat}. Define smooth cutoff functions $\{ \eta_k \}_{k \in \N}$ such that
\[
  \begin{split}
  \eta_k \equiv 1 \text{ in } Q_{7r_k/8}, \quad
  \eta_k \equiv 0 \text{ in } \R^4 \times (-\infty, 0) \backslash Q_k, \quad
  |\nabla \eta_k| \leq C' r_k^{-1}.
  \end{split}
\]
Then define $\varphi_k := \phi_n(\eta_k-\eta_{k+1})$ for $1\leq k \leq n-1$ and $\varphi_n:=\phi_n \eta_n$. It is easy to check the bound
\begin{equation} \label{extrabound}
  |\nabla \varphi_k| = |\phi_n \nabla \eta_k + \eta_k \nabla \phi_n| \leq C_4 C' r_k^{-5}
  \quad \text{for any }k \leq n
\end{equation}
and the fact $\phi_n = \sum_{k=1}^{n} \varphi_k$.
\par
We use $\phi_n$ as the cutoff function in the local energy inequality \eqref{evo_local_energy_1} and choose the functions $\gamma_k=\tilde{p}_{r_k}$. This yields
\begin{equation}
  \begin{split}
  \limsup_{ k \rightarrow \infty} \sup_t \int_{B_{1/2}} |u_k|^2 \phi_n dx + \iint_{Q_{1/2}} \phi_n &\big( |\nabla u|^2 dxdt + d\lambda \big) \\
     &\leq I_1 + 3I_2 + 2I_3 +I_4,
  \end{split}
\end{equation}
where
\[
  \begin{split}
  I_1 &= \iint_{Q_{1/2}}  |u|^2 |\partial_t \phi_n + \Delta \phi_n| dxdt \\
  I_2 &= \sum_{k=1}^n \iint_{Q_{1/2}} |\nabla \varphi_k| \big( |u|^3 dxdt + d\omega \big) \\
  I_3 &= \sum_{k=1}^n \iint_{Q_{1/2}} |\nabla \varphi_k| |p-\tilde{p}_{r_k}|^{3/2} dxdt \\
  I_4 &= \iint_{Q_{1/2}}  |u| |f| |\phi_n| dxdt.
  \end{split}
\]
With the bounds in \Cref{backwardheat}, we can deduce
\begin{equation} \label{claim2eq1}
C_4^{-1} r_{n+1}^{-2} \big(A(r_{n+1}) + \delta(r_{n+1}) + \delta_c(r_{n+1})\big) \leq I_1 + 3I_2 + 2I_3 + I_4
\end{equation}
\par
For $I_1$, we use the bounds in \Cref{backwardheat}, H{\"o}lder's inequality and the initial smallness condition \eqref{smallness},
\[
  \begin{split}
  I_1 & \leq C_4 \iint_{Q_{1/2}}  |u|^2 dxdt
      \leq C_4 \Big( \iint_{Q_{1/2}} 1 dxdt \Big)^{1/3} \Big( \iint_{Q_{1/2}} |u|^3 dxdt \Big)^{2/3} \\
    & \leq C_4 \Big( \iint_{Q_1} |u|^3 dxdt \Big)^{2/3}
      \leq C_4 \varepsilon^{2/3}
  \end{split}.
\]
\par
For $I_2$, we need to decompose the integral over $Q_{1/2}$ into integrals over parabolic rings. Then for each subintegral we use the bounds in \Cref{backwardheat} and our induction hypothesis $\{\eqref{ainduction}_k\}_{1 \leq k \leq n}$ to obtain
\[  
  \begin{split}
  I_2 &= \sum^n_{k=2} \iint_{Q_{k-1} \backslash Q_k} | \nabla \varphi_k| \big( |u|^3dxdt + d\omega \big)
      + \iint_{Q_n} |\nabla \varphi_k| \big( |u|^3dxdt + d\omega \big) \\
    &\leq C_4C' \sum^n_{k=2} r_k^{-5} \iint_{Q_{k-1} \backslash Q_k} \big( |u|^3dxdt + d\omega \big)
      + C_4C' r_n^{-5} \iint_{Q_n} \big( |u|^3dxdt + d\omega \big) \\
    &\leq C_4C' 2^5 \sum^n_{k=2} r_{k-1}^{-5} \iint_{Q_{k-1}} \big( |u|^3dxdt + d\omega \big)
      + C_4C' r_n^{-5} \iint_{Q_n} \big( |u|^3dxdt + d\omega \big) \\
    &\leq C_4C' 2^5 \sum^n_{k=1} r_k \varepsilon^{2/3}.
  \end{split}
\]
\par
We estimate $I_4$ and $I_3$ in a similar way, doing the decomposition and using the bounds in \Cref{backwardheat}, H{\"o}lder's inequality, the initial smallness condition \eqref{smallness} and the induction hypothesis $\{\eqref{ainduction}_k\}_{1 \leq k \leq n}$,
\[
  \begin{split}
  I_4 \leq & \sum^n_{k=2} \iint_{Q_{k-1} \backslash Q_k} |u||f||\varphi_k| dxdt
      + \iint_{Q_n} |u||f||\varphi_k| dxdt  \\
    \leq & C_4C' \sum^n_{k=1} r_k^{-4} \Big( \iint_{Q_k} |u|^3 dxdt \Big)^{1/3} 
      \Big( \iint_{Q_k} |f|^q dxdt \Big)^{1/q} \Big( \iint_{Q_k} 1 dxdt \Big)^{2/3-1/q} \\
    \leq & C_4C' \sum^n_{k=1} r_k^{2-6/q} \varepsilon^{2/9} \kappa^{1/q},
  \end{split}
\]
\[
  \begin{split}
  I_3 
    \leq & C_4C'\sum^n_{k=2} r_k^{-5} \iint_{Q_{k-1} \backslash Q_k} |p-\tilde{p}_{r_k}|^{3/2} dxdt
      + C_4C' r_n^{-5} \iint_{Q_n} |p-\tilde{p}_{r_n}|^{3/2} dxdt\\
    \leq & C_4C' \sum_{k=1}^n r_k^{1/2} \varepsilon^{2/3},
  \end{split}
\]
respectively.
\par
Now we can combine the estimates for $I_1,I_2,I_3,I_4$ and the inequality \eqref{claim2eq1},
\[
  \begin{split}
  A&(r_{n+1}) + \delta(r_{n+1}) + \delta_c(r_{n+1})  \\
    &\leq C_4^2 r_{n+1}^2 \Big( \varepsilon^{2/3} 
      + 96C' \sum^n_{k=1} r_k \varepsilon^{2/3} + 2C' \sum_{k=1}^n r_k^{1/2} \varepsilon^{2/3} 
      + C' \sum^n_{k=1} r_k^{2-6/q} \varepsilon^{2/9} \kappa^{1/q} \Big).
  \end{split}
\]
Since $q>3$, we can choose the constants $C_B$ and $\kappa$ as
\[
  \begin{split}
  C_B &= C_4^2 \Big( 1
      + 96C' \sum^n_{k=1} r_k + 2C' \sum_{k=1}^n r_k^{1/2}
      + C' \sum^n_{k=1} r_k^{2-6/q} \Big), \\
  \kappa &= \varepsilon^{4q/9}.
  \end{split}
\]
Because $C'$ and $C_4$ are absolute, $C_B$ only depends on $q$. Then it yields
\[ A(r_{n+1}) + \delta (r_{n+1}) + \delta_c (r_{n+1}) \leq C_B \varepsilon^{2/3} r_{n+1}^2. \]
Since this argument and the constants $C_B,\kappa$ are uniform for all $(x_0,t_0) \in Q_{1/2}(0,0)$, we can deduce that $\eqref{binduction}_{n+1}$ holds.
\end{proof}
\vspace{2mm}
\noindent 
\textbf{Claim 3:} \textit{$\{\eqref{binduction}_k\}_{2 \leq k \leq n}$ implies $\eqref{ainduction}_n$.} \par
\begin{proof}[Proof of Claim 3]
For simplicity, let $(x_0,t_0)=(0,0)$. The interpolation inequality \Cref{interpolation} yields for any $2 \leq k \leq n$,
\begin{equation} \label{cknprop_c3_eq1}
  G(r_k) + G_c(r_k) \leq C_1 A^{3/2}(r_k) + C_1 \delta^{3/2}(r_k) + C_1 \delta_c^{3/2}(r_k) \leq C_1C_B \varepsilon r_k^3.
\end{equation}
Let $\rho=r_1, r=r_n$, \Cref{pressurepreestimate} yields
\begin{equation} \label{cknprop_c3_eq2}
  \begin{split}
  L'(r_n) \leq & C_2 r_n^{-5/2} \iint_{Q_{n-1}} |u|^3 dxdt 
    + C_2 r_n^5 \Big( \underset{-r_n^2<t<0}{\sup} \int_{2r_n<|y|<r_1} \frac{|u|^2}{|y|^5} dy \Big)^{3/2} \\
    & + C_2 \frac{r_n^3}{r_1^{11/2}} \iint_{Q_{r_1}} \big(|u|^3 + |p|^{3/2}\big) dxdt \\
  \leq& 8C_2 r_n^{1/2} G(r_{n-1}) + C_2 r_n^5 
    \Big( \underset{-r_n^2<t<0}{\sup} \int_{2r_n<|y|<r_1} \frac{|u|^2}{|y|^5} dy \Big)^{3/2} 
    + C_2 \varepsilon r_n^3 \\
  \leq& C_2 \big( 8C_1C_B^{3/2}r_n^{1/2}+1 \big) \varepsilon r_n^3 + C_2 r_n^5 
    \Big( \underset{-r_n^2<t<0}{\sup} \int_{2r_n<|y|<r_1} \frac{|u|^2}{|y|^5} dy \Big)^{3/2}.
  \end{split}
\end{equation}
In the first line of \eqref{cknprop_c3_eq2}, for the first term, we use the interpolation inequality in \Cref{interpolation} and our induction hypothesis $\eqref{binduction}_{n-1}$. For the third term, we use the initial smallness condition \eqref{smallness}. For the second term, we decompose this integral into integrals over rings and estimate it using induction hypothesis $\{\eqref{binduction}_k\}_{2 \leq k \leq n}$,
\begin{equation} \label{cknprop_c3_eq3}
  \begin{split}
  \underset{-r_n^2<t<0}{\sup} \int_{2r_n<|y|<r_1} \frac{|u|^2}{|y|^5} dy
    &\leq \sum^{n-1}_{k=2} \underset{-r_{k-1}^2<t<0}{\sup} \int_{r_k<|y|<r_{k-1}} \frac{|u|^2}{|y|^5} dy \\
    &\leq 4\sum^{n-1}_{k=2} r_k^{-3} A(r_{k-1}) \\
    &\leq 16C_B \varepsilon^{2/3} \sum^{n-1}_{k=2} r_k^{-1} \\
    &\leq 16C_B \varepsilon^{2/3} r_n^{-1}.
  \end{split}
\end{equation}
In the second inequality we use $|y|^{-5} \leq r_k^{-5}$ when $r_k<|y|<r_{k-1}$. The third inequality follows from our induction hypothesis $\{\eqref{binduction}_k\}_{2 \leq k \leq n}$. \par
Hence, from \eqref{cknprop_c3_eq1}, \eqref{cknprop_c3_eq2} and \eqref{cknprop_c3_eq3}, we can deduce that
\[ G(r_n) + G_c(r_n) + L'(r_n) \leq \big[ C_2 ( 8C_1C_B^{3/2}r_n^{1/2}+1 + 64C_B^{3/2}r_n^{1/2} ) + C_1C_B \big] \varepsilon r_n^3. \]
Now we impose the second condition on $\varepsilon>0$,
\[ \big[ C_2 ( 8C_1C_B^{3/2}r_n^{1/2}+1 + 64C_B^{3/2}r_n^{1/2} ) + C_1C_B \big] \varepsilon^{1/3} < 1. \]
It yields
\[  G(r_n)+G_c(r_n)+L'(r_n) \leq \varepsilon^{2/3} r_n^3. \]
Because $C_1$ and $C_2$ are absolute and $C_B$ only depends on $q$, the choice of $\varepsilon$ is uniform for any $(x_0,t_0) \in Q_{1/2}(0,0)$.
\end{proof}
Now we can deduce that $\eqref{binduction}_k$ holds for any $k \geq 2$. This gives
\[ \underset{t_0-r_n^2<t<t_0}{\sup} r_n^{-4} \int_{B_{r_n}(x_0)} |u|^2 dx
    \leq \limsup_{k \rightarrow \infty}\underset{t_0-r_n^2<t<t_0}{\sup} r_n^{-4} \int_{B_{r_n}(x_0)} |u_k|^2 dx 
    \leq C_B \varepsilon^{2/3} \]
for any $(x_0,t_0) \in Q_{1/2}(0,0)$ and $n \geq 2$. Hence
\[ |u(x,t)| \leq C_B^{1/2} \varepsilon^{1/3}, \]
given that $(x,t) \in Q_{1/2}(0,0)$ is a Lebesgue point of $u$.
\end{proof}

\begin{remark}
Observe that the above proof uses the term $L'$. If instead of $L'$, we were to carry out the estimate \eqref{cknprop_c3_eq2} with $L$, we would obtain $\varepsilon r^{5/2}_n$ when estimating the term in the second line of \eqref{cknprop_c3_eq2}, which cannot be bounded by $\varepsilon^{2/3} r^3_n$ uniformly in $n \in \N$.
\end{remark}

\par

The second partial regularity result corresponds to a version of Proposition 2 in \cite{caffarelli1982partial} in space dimension 4 with concentration measures. We use an idea from Lin's work \cite{lin1998new} where he gave a simpler proof for the results in \cite{caffarelli1982partial}. As a consequence, we are able improve Scheffer's result in \cite{scheffer1978navier}, to show that the 2-dimensional parabolic Hausdorff measure of the singular set of $u$ in space dimension 4 is finite.

\begin{proposition} \label{ckn2_nonsta}
There exists an absolute constant $\tau >0$ with the following property. Suppose that $(u,p,\lambda,\omega)$ is a weak solution set of the Navier-Stokes equations \eqref{navierstokes} in some cylinder $Q_{\rho}(x_0,t_0)$ in space dimension $4$ with $f \in L^q_{\text{loc}}(Q_{\rho}(x_0,t_0))$ for some $q>3$. If
\[ \underset{r \rightarrow 0}{\lim \sup} \frac{1}{r^2} \iint_{Q_r(x_0,t_0)} \big( |\nabla u|^2 dxdt + d\lambda \big) \leq \tau, \]
then $\|u\|_{ L^{\infty}(Q_{r_0}(x_0,t_0)) } < Cr_0^{-1}$ for some $0<r_0<\rho$. Note that $C=C(\varepsilon,q)$ depends on $\varepsilon$ and $q$ in \Cref{cknprop}.
\end{proposition}

\begin{proof}
Because $q>3$, there holds $3q-6>3$. Then there exists $r_1>0$, such that for any $0 < r \leq r_1$,
\[ F_1(r,x_0,t_0) = r^{3q-6} \iint_{Q_r(x_0,t_0)}  |f|^{q} dxdt \leq \kappa. \]
\vspace{2mm}
\noindent 
\textbf{Claim 1:} \textit{For some $0 < r_2 \leq r_1$,}
\begin{equation} \label{ckn2_nonsta_eq0}
  r_2^{-3} \iint_{Q_{r_2}(x_0,t_0)} \Big( |u|^3 + |p|^{3/2} \Big) dxdt + r_2^{-3} \iint_{Q_{r_2}} d\omega \leq \varepsilon.
\end{equation}
If this claim holds, \Cref{cknprop} yields $\|u\|_{ L^{\infty}(Q_{r_2/2}(x_0,t_0)) } < Cr_2^{-1}$ immediately. 

\begin{proof}[Proof of Claim 1]
We use the local energy inequality \eqref{evo_local_energy_2} to derive the smallness condition \eqref{ckn2_nonsta_eq0}. Fix $r \in (0, \rho)$, $\theta \in (0,\frac{1}{2}]$ and $(x_0,t_0)=(0,0)$. We choose cutoff function $\phi_{\theta}$ as stated in \Cref{alter_backwardheat}. Then from \eqref{evo_local_energy_2} we deduce
\begin{equation} \label{ckn2_nonsta_eq1}
  \begin{split}
  \frac{1}{C_5(\theta r)^2} \limsup_{k \rightarrow \infty}\underset{-(\theta r)^2<t<0}{\sup} \int_{B_{\theta r}} |u_k|^2 dx 
  + \frac{1}{C_5(\theta r)^2} \iint_{Q_{\theta r}} \big( |\nabla u|^2 dxdt + d\lambda \big) \\
  \leq I_1 + I_2 + 3I_2' + 2I_3 + I_4,
  \end{split}
\end{equation}
where
\[
  \begin{split}
  I_1 &= (\theta r)^2\iint_{Q_r} |u|^2(\partial_t \phi_{\theta} + \Delta \phi_{\theta}) dxdt, \\
  I_2 &= (\theta r)^2\iint_{Q_r} |u|^2 u \cdot \nabla \phi_{\theta} dxdt, \\
  I_2' &= (\theta r)^2\iint_{Q_r} |\nabla \phi_{\theta}| \big( |u-\tilde{u}_r|^3 dxdt + d\omega \big), \\
  I_3 &= (\theta r)^2\iint_{Q_r} p u \cdot\nabla \phi_{\theta} dxdt, \\
  I_4 &= (\theta r)^2\iint_{Q_r} f \cdot u \phi_{\theta} dxdt.
  \end{split}
\] \par
For $I_1,I_3$ and $I_4$, we simply use H{\"o}lder's inequality and the bounds in \Cref{alter_backwardheat}. For $I_2'$, we use the interpolation inequalities in \Cref{interpolation}. Thus, we obtain
\begin{equation} \label{ckn2_nonsta_eq2}
  \begin{split}
  I_1 &\leq C_5 \theta^2 G^{2/3}(r), \\
  I_3 &\leq C_5 \theta^{-3} K^{2/3}(r) G^{1/3}(r), \\
  I_4 &\leq C_5 \theta^{-2} F_2^{1/2}(r) G^{1/3}(r), \\
  I_2' &\leq C_5 \theta^{-3} [H(r)+G_c(r)].
  \end{split}
\end{equation} \par
For $I_2$, we reduce the estimates on $u$ to estimates on $u-\tilde{u}_r$, then we use the interpolation inequality in \Cref{interpolation}. Note that $|u|^2 \in L_t^2W^{1,1}_x(Q_r)$, because $u \in L_t^2H^1_x(Q_r) \cap L_t^{\infty}L^2_x(Q_r)$. We have
\begin{equation} \label{ckn2_nonsta_eq3}
  \begin{split}
  I_2  =  & (\theta r)^2 \iint_{Q_r} |u|^2 (u-\tilde{u}_r) \cdot \nabla \phi_{\theta} dxdt
      + (\theta r)^2 \iint_{Q_r} |u|^2 \tilde{u}_r \cdot \nabla \phi_{\theta} dxdt \\
     =  & (\theta r)^2 \iint_{Q_r} |u-\tilde{u}_r|^2 (u-\tilde{u}_r) \cdot \nabla \phi_{\theta} dxdt
      + 2 (\theta r)^2 \iint_{Q_r} u \cdot \tilde{u}_r (u-\tilde{u}_r) \cdot \nabla \phi_{\theta} dxdt \\
      & + (\theta r)^2 \iint_{Q_r} |\tilde{u}_r|^2 (u-\tilde{u}_r) \cdot \nabla \phi_{\theta} dxdt
      + (\theta r)^2 \iint_{Q_r} |u|^2 \tilde{u}_r \cdot \nabla \phi_{\theta} dxdt \\
    \leq& C_5(\theta r)^{-3} \iint_{Q_r} |u-\tilde{u}_r|^3 dxdt
      - 2 (\theta r)^2 \iint_{Q_r} \big[\big((u-\tilde{u}_r) \cdot \nabla\big) u\big] \cdot \tilde{u}_r \phi_{\theta} dxdt \\
      & - 2(\theta r)^2 \iint_{Q_r} u \cdot \big((\tilde{u}_r \cdot \nabla) u\big) \phi_{\theta} dxdt,
  \end{split}
\end{equation}
where we use the bounds in \Cref{alter_backwardheat} and integration by parts. Notice that $u$ and $u-\tilde{u}_r$ are divergence-free. Furthermore, we have
\begin{equation} \label{ckn2_nonsta_eq4}
  \begin{split}
  \iint_{Q_r} \big[&\big((u-\tilde{u}_r) \cdot \nabla\big) u\big] \cdot \tilde{u}_r \phi_{\theta} dxdt
  \leq C_5 (\theta r)^{-4} \int_{-r^2}^0 r^{-4} \int_{B_r} |u|dx \int_{B_r} |u-\tilde{u}_r| |\nabla u| dx dt \\
  &\leq C_5 \theta^{-4} r^{-16/3} \int_{-r^2}^0 \Big(\int_{B_r}|u|^3dx\Big)^{1/3} \Big(\int_{B_r}|u-\tilde{u}_r|^2dx\Big)^{1/2} \Big(\int_{B_r}|\nabla u|^2dx\Big)^{1/2} dt \\
  &\leq C_5 \theta^{-4} r^{-5} \Big(\iint_{Q_r}|u|^3dxdt\Big)^{1/3} \Big(\iint_{Q_r}|\nabla u|^2dxdt\Big)^{1/2} 
    \Big( \underset{-r^2<t<0}{\sup} \int_{B_r}|u|^2dx \Big)^{1/2}.
  \end{split}
\end{equation}
The first inequality follows from the bounds in \Cref{alter_backwardheat}. The remaining inequalities follow from H{\"o}lder's inequality. Analogously,
\begin{equation} \label{ckn2_nonsta_eq5}
  \begin{split}
  \iint_{Q_r} u \cdot & \big((\tilde{u}_r \cdot \nabla) u\big) \phi_{\theta} dxdt \\
  &\leq C_5 \theta^{-4} r^{-5} \Big(\iint_{Q_r}|u|^3dxdt\Big)^{1/3} \Big(\iint_{Q_r}|\nabla u|^2dxdt\Big)^{1/2} 
    \Big( \underset{-r^2<t<0}{\sup} \int_{B_r}|u|^2dx \Big)^{1/2}.
  \end{split}
\end{equation}
Now, from \eqref{ckn2_nonsta_eq3}, \eqref{ckn2_nonsta_eq4} and \eqref{ckn2_nonsta_eq5}, we deduce
\begin{equation} \label{ckn2_nonsta_eq6}
  \begin{split}
  I_2 \leq 2C_5 \theta^{-2} G^{1/3}(r) A^{1/2}(r) \delta^{1/2}(r) + C_5 \theta^{-3}H(r).
  \end{split}
\end{equation} \par
Now, we are in a position to plug \eqref{ckn2_nonsta_eq2} and \eqref{ckn2_nonsta_eq6} into \eqref{ckn2_nonsta_eq1} and to invoke the interpolation inequality in \Cref{interpolation},
\begin{equation} \label{ckn2_nonsta_eq7}
  \begin{split}
  A(\theta r) + 2 [\delta(\theta r)+\delta_c(\theta r)] \lesssim& C_5^2
    \Big( \theta^2 G^{2/3}(r) + \theta^{-3} K^{2/3}(r) G^{1/3}(r) \\
      &+ 2\theta^{-2} F_2^{1/2}(r) G^{1/3}(r) + \theta^{-3} [H(r)+G_c(r)] \\
      &+ 3\theta^{-2} G^{1/3}(r) A^{1/2}(r) \delta^{1/2}(r) \Big) \\
    \leq& C_5^2 \Big( 4\theta^2 G^{2/3}(r) + \theta^{-8} K^{4/3}(r) + \theta^{-6} F_2(r) \\
      &+ C_1\theta^{-3} A^{1/2}(r) [\delta(r)+\delta_c(r)] + \frac{9}{4}\theta^{-6} A(r) \delta(r) \Big) \\
    \leq& C_5^2 \Big[ 4C_1\theta^2 A(r) + 5C_1\theta^2 [\delta(r)+\delta_c(r)] + C_1\theta^{-8} A(r) [\delta(r)+\delta_c(r)] \\
      & + \frac{9}{4}\theta^{-6} A(r) \delta(r) + \theta^{-8} K^{4/3}(r) + \theta^{-6} F_2(r) \Big].
  \end{split}
\end{equation}
In the second inequality, we use Young's inequality to move $\theta G^{1/3}(r)$ to the first term. In the third inequality, we use the interpolation inequality in \Cref{interpolation}. And Young's inequality moves $\theta^{-3} A^{1/2}(r) \delta^{1/2}(r)$ to the term $4\theta^{-6} A(r) \delta(r)$. \par
On the other hand, with \Cref{pressurepreestimate_2} we deduce
\begin{equation} \label{ckn2_nonsta_eq8}
  \begin{split}
  K^{4/3}(\theta r) &\lesssim C_1^{4/3}C_3^{4/3} \theta^{-4} A^{2/3}(r) \delta^{4/3}(r) 
    + C_3^{4/3} \theta^{4/3} K^{4/3}(r) \\
  &\lesssim C_1^2 C_3^2 \theta^{-12} A(r) \delta(r) + \theta^{12} \delta^2(r) + C_3^{4/3} \theta^{4/3} K^{4/3}(r).
  \end{split}
\end{equation} \par
Taking the sum of \eqref{ckn2_nonsta_eq7} and $\theta^{-9} \times \eqref{ckn2_nonsta_eq8}$ gives
\begin{equation} \label{ckn2_nonsta_eq9}
  \begin{split}
  A(\theta r) + &\theta^{-9} K^{4/3}(\theta r) + 2[\delta(\theta r) + \delta_c(\theta r)] \\
  \lesssim& 4C_1C_5^2\theta^2 A(r) + 5C_1C_5^2\theta^2[\delta(r)+\delta_c(r)] + \theta^3 \delta^2(r) \\ 
    & + \Big(\frac{9}{4}C_5^2+C_1^2C_3^2\theta^{-15}+C_1C_5^2 \theta^{-2}\Big)\theta^{-6} A(r) [\delta(r)+\delta_c(r)] \\
    & + (C_5^2 \theta + C_3^{4/3} \theta^{4/3})\theta^{-9} K^{4/3}(r) + C_5^2\theta^{-6} F_2(r).
  \end{split}
\end{equation}
Since $C_1,C_3$ and $C_5$ are absolute positive constants, we can fix $\theta \in (0,\frac{1}{2}]$ such that
\begin{equation} \label{ckn2_nonsta_eq10}
  \begin{split}
  4C_1C_5^2\theta^2 &\lesssim \frac{1}{4}, \\
  C_5^2 \theta + C_3^{4/3} \theta^{4/3} &\lesssim \frac{1}{2}, \\
  \big(5C_1C_5^2 + 1\big)\theta^2 &\lesssim \frac{1}{8}.
  \end{split}
\end{equation}
Then we choose $\tau \in (0,1)$ such that
\begin{equation} \label{ckn2_nonsta_eq10}
  \begin{split}
  \Big(\frac{9}{4}C_5^2+C_1^2C_3^2\theta^{-10}+C_1C_5^2 \theta^{-2}\Big)\theta^{-6} \cdot 2\tau \lesssim \frac{1}{4}.
  \end{split}
\end{equation}
Because
\[ \underset{r \rightarrow 0}{\lim \sup} \text{ } \delta(r) + \delta_c(r) \leq \tau,
  \quad \underset{r \rightarrow 0}{\lim} \text{ } F_2(r) = 0, \]
we can choose $r'>0$ such that for any $0 < r \leq r'$,
\[ \delta(r) + \delta_c(r) \leq 2\tau, \quad C_5^2\theta^{-6} F_2(r) \leq \frac{\tau}{4}. \]
\par
Let $E(r) = A(r) + \theta^{-9} K^{4/3}(r) + 2[\delta(r) + \delta_c(r)]$, then \eqref{ckn2_nonsta_eq9} yields for any $0 < r \leq r'$,
\begin{equation} \label{ckn2_nonsta_eq11}
  \begin{split}
  E(\theta r) \lesssim \frac{1}{2} E(r) + \frac{\tau}{2}.
  \end{split}
\end{equation}
Iterating this inequality yields for any $k \in \N$
\[ A(\theta^k r') + \theta^{-9} K^{4/3}(\theta^k r') + 2[\delta(\theta^k r') + \delta_c(\theta^k r')] = E(\theta^k r') \lesssim \frac{1}{2^k} E(r') + \tau. \]
Then there exists some $r_2>0$ such that
\[ A(r_2) + \delta(r_2) + \delta_c(r_2) + \theta^{-9} K^{4/3}(r_2) \lesssim 4 \tau. \]
Again by the interpolation inequality in \Cref{interpolation}, we can bound $G(r_2) + G_c(r_2)$ with $A(r_2) + \delta(r_2) + \delta_c(r_2)$. Then we can impose another condition on $\tau$ to ensure \eqref{ckn2_nonsta_eq0}. This additional condition on the choice of $\tau$ depends on $\varepsilon, C_1$ and $\theta$, so it does not produce any circular reasoning. This concludes the proof.
\end{proof}
\end{proof}

Now, we give definitions to singular set of weak solution set and parabolic Hausdorff measure of a space-time set.

\begin{definition} \label{singularset_nonsta}
Suppose that $(u,p,\lambda,\omega)$ is a weak solution set of the Navier-Stokes equation in $\R^4 \times [0,T]$. A point $(x,t) \in \R^4 \times (0,T]$ is called a regular point if there exists $r>0$ such that $u \in L^{\infty}(Q_r(x,t))$. Otherwise, $(x,t)$ is called a singular point. The singular set is the set of all singular points.
\end{definition}

\begin{definition} \label{parabolichausdorff}
Given a set $D \subset \R^4 \times \R$, for a fixed positive real number $s$, $s$-dimensional parabolic Hausdorff measure is defined as
\[ \mathcal{P}^s (D) = \lim_{\delta \rightarrow 0^+} \mathcal{P}^s_{\delta} (D), \]
where
\[ \mathcal{P}^s_{\delta} (D) 
  = \inf \Big\{ \sum_{i=1}^{\infty} r_i^s \Big| 
  D \subset \bigcup_{i \in \N^+} Q^{\ast}_{r_i}(x_0,t_0), 0 < r_i < \delta, (x_0,t_0) \in \R^4 \times \R \Big\}. \]
Here, $Q^{\ast}_r(x,t)$ is centered parabolic cylinder defined by
\[ Q^{\ast}_r(x,t) := B_r(x) \times \Big( t-\frac{r^2}{2}, t+\frac{r^2}{2} \Big). \]
\end{definition}

\Cref{maintheorem_nonsta} follows from \Cref{preparation_existence_nonsta}, \Cref{energy_convergence}, \Cref{ckn2_nonsta} and the following standard covering argument.
\begin{proof}[Proof of \Cref{maintheorem_nonsta}]
We assume $S$ is bounded and $S \subset B_{\rho_0} \times [0,T]$ for some $\rho_0>0$. Let $D':= \overline{B_{\rho_0+1}(\R^4)} \times [0,T]$. Let $V$ be a parabolic neighborhood (neighborhood given by parabolic cylinders) of $S$ in $D'$ and fix $\delta > 0$. According to \Cref{ckn2_nonsta}, for each $(x,t) \in S$, we choose $Q_r(x,t) \subset V$ with $r < \delta$ such that
\[ r^{-2} \iint_{Q_r(x,t)} \big( |\nabla u|^2 dxdt + d\lambda \big)
  > \tau. \]
\par

Because $S$ is bounded, we can use Vitali Covering lemma to obtain a family of disjoint parabolic cylinders $\{Q_{r_i}(x_i,t_i)\}_{i \in \Lambda}$ such that
\[ S \subset \bigcup_{i \in \Lambda} Q_{5r_i}(x_i,t_i). \]
Here $\Lambda$ is a finite set. Then
\[ \sum_{i \in \Lambda} r_i^2 
  \leq \frac{1}{\tau} \sum_{i \in \Lambda} \iint_{Q_{r_i}(x_i,t_i)} \big( |\nabla u|^2 dxdt + d\lambda \big)
  \leq \frac{1}{\tau} \iint_V \big( |\nabla u|^2 dxdt + d\lambda \big). \]
Since $\delta$ is arbitrary, we know the Lebesgue measure of $S$ is zero and
\begin{equation}
  \label{nsmt1_eq1}
  \mathcal{P}^2(S) \leq \frac{5}{\tau} \iint_V \big( |\nabla u|^2 dxdt + d\lambda \big).
\end{equation}
\par
In case that $S$ is unbounded, we look at $S \cap B_r \times [0,T]$ with $r \rightarrow \infty$, then $\mathcal{P}^2(S\cap B_r \times [0,T])$ is bounded uniformly in $r$, which concludes our proof.
\end{proof}

\begin{remark}
Dong, Gu \cite{dong2014partial} and Wang, Wu \cite{wang2014unified} proved that suitable weak solutions satisfy $\mathcal{P}^2(S) =0$, but they were not able to show that such solutions exist. Here, we can only prove $\mathcal{P}^2(S) < \infty$, since the presence of the concentration measures leads to nontriviality of the Hausdorff measure of the singular set $S$ in this covering argument.
\end{remark}

\section{acknowledgements}
I would like to express my gratitude to my advisor Michael Struwe for reading my draft carefully and giving me helpful feedback. I would also like to thank the referees for very useful comments to improve this paper.

\appendix

\section{Fractional power of nonnegative smooth function} \label{app1}

In this appendix, we prove that certain fractional power of any nonnegative smooth function is Lipschitz continuous, which is a direct consequence of the following lemma proved by Fefferman and Phong \cite{fefferman1978positivity}.

\begin{lemma}[Fefferman and Phong \cite{fefferman1978positivity}; Lemma 4, Guan \cite{guan1997c_2}] \label{guan}
If $f:\R^n \rightarrow \R$ is a $C^{3,1}$ nonnegative function, with $\|f\|_{C^4} \leq A$, then there is $N \in \N$ (only depends on $n$) and functions $g_1, g_2, \ldots, g_N \in C^{1,1}$, with $\|g_j\|_{C^2} \leq C$, such that
\[ f=\sum_{j=1}^{N} g_j^2, \]
where the constant $C$ depends on $n$ and $A$.
\end{lemma}

\begin{corollary} \label{nonnegative_smooth}
Suppose $f:\R^n \rightarrow \R$ is a $C^{3,1}$ nonnegative function, then $h:=f^{\alpha}$ is Lipschitz continuous for any $\alpha \in [\frac{1}{2},1]$.
\end{corollary}

\begin{proof}
This result follows from \Cref{guan} and the following bound,
\[
  \begin{split}
  |h'| &= \frac{2\alpha \big|\sum_{j=1}^{N}g_jg'_j\big|} {\big(\sum_{i=1}^{N} g_i^2\big)^{1-\alpha}} \\
  &\leq \sum_{j=1}^{N} \frac{2\alpha \big|g_j^{2\alpha-1}g'_j\big|} {\big(\sum_{i=1}^{N} (g_i/g_j)^2\big)^{1-\alpha}} \\
  &\leq \sum_{j=1}^{N} 2\alpha \big|g_j^{2\alpha-1}g'_j\big|.
  \end{split}
\]
\end{proof}




\end{document}